\documentclass[10pt,leqno]{amsart}
\usepackage{indentfirst, amssymb,amsmath,amsthm,enumerate}
\usepackage{graphics,graphicx,tikz,epstopdf}
\usepackage{multicol}
\usepackage{hyperref}
\newtheorem*{theoA}{Theorem A}
\newtheorem*{theoB}{Theorem B}
\newtheorem*{theoC}{Theorem C}
\newtheorem*{theoD}{Theorem D}

\newtheorem{theo}{Theorem}[section]
\newtheorem{lem}{Lemma}[section]

\newtheorem{exm}{Example}[section]
\newtheorem{defi}{Definition}[section]
\newtheorem{rem}{Remark}[section]
\newtheorem{ques}{Question}[section]

\newtheorem{proposition}{Proposition}

\newcommand{\be}{\begin{equation}}
	\newcommand{\ee}{\end{equation}}
\newcommand{\beas}{\begin{eqnarray*}}
	\newcommand{\eeas}{\end{eqnarray*}}
\newcommand{\bea}{\begin{eqnarray}}
	\newcommand{\eea}{\end{eqnarray}}

\numberwithin{equation}{section}

\begin{document}
	\title[On the characterization of uniqueness polynomials  ...]{On the characterization of uniqueness polynomials :
	\\Both Complex and P-adic Versions} 
		\author{Pratap Basak}
	\address{ Department of Mathematics, Cooch Behar Panchanan Barma University, West Bengal 736101, India.}
	\email{iampratapbasak@yahoo.com, pratapbasak6@gmail.com}
		\author{Sanjay Mallick}
		\address{ Department of Mathematics, Cooch Behar Panchanan Barma University, West Bengal 736101, India.}
		\email{sanjay.mallick1986@gmail.com, smallick.ku@gmail.com}
	\maketitle
	\let\thefootnote\relax
	\footnotetext{Typeset by \LaTeX-2e\par Corresponding Author:- Sanjay Mallick}
	\maketitle
	\begin{abstract}

The problem ``A general characterization of uniqueness polynomial  for non-critically injective polynomials" has been remained open since the last two decades. In this paper, we explore this open problem. To this end, we  initiate a new approach  that also includes critically injective polynomials. We provide this characterization for both the complex and p-adic cases. We also provide various examples as an application of our results along with the  verification of the existing examples. Consequently, we find examples of unique range sets generated by non-critically injective polynomials with least cardinalities achieved so far and one of these results is sharp with respect to all the available formulas in the literature. Furthermore, we  cover the part of  least degree uniqueness polynomials. 
In this part, we also provide some sharp bounds. 
	\end{abstract}
	\section{\Large Introduction and Definitions}
At the outset, we define by  ${ \mathbb{K}}$  an algebraically closed field of characteristic zero and complete for a non-Archimedean absolute value. Let ${ \mathbb{L}}$ be a field, which is either ${ \mathbb{C}}$ or ${ \mathbb{K}}$. We denote by ${\mathcal{A}( \mathbb{L})}$ the ring of entire functions in ${ \mathbb{L}}$, and by ${\mathcal{M}( \mathbb{L})}$ the field of meromorphic functions in ${ \mathbb{L}}$. 
\par\vspace{0.1in}
	A non-constant polynomial ${P(z)}$ in $\mathbb{L}$ is called a uniqueness polynomial for meromorphic (entire) functions if for any two non-constant meromorphic (entire) functions ${f,g\in \mathcal{M}(\mathbb{L})}$ $( \mathcal{A}(\mathbb{L}))$, the identity ${P(f)=P(g)}$ implies ${f=g}$. In short, we denote it by UPM  (UPE).
\par\vspace{0.1in} The notion of uniqueness polynomials originated from the characterization of unique range sets which is defined as follows. 
\par\vspace{0.1in} For two non-constant meromorphic (entire) functions $f, g\in \mathcal{M}(\mathbb{L})$ $(\mathcal{A}(\mathbb{L}))$ and a set $S\subseteq \mathbb{L}$ if we find that $f^{-1}(S)=g^{-1}(S)$ implies $f= g$, then we say $S$ is a unique range set for meromorphic (entire) functions or URSM (URSE), where the elements in the pre-image set are counted according to their multiplicities. 
\par If  the multiplicities are not taken into account, then we say $S$ is  a unique range set for meromorphic (entire) functions ignoring multiplicities or URSM-IM (URSE-IM).  
\par\vspace{0.1in}
Since the inception of unique range sets in 1976 \cite{gross} to till date, a rigorous research have been performed by various authors \cite{Mallick-cmft,Frank Rainders,ccy-kodai95,Hxy-Bul,mallick-filo} in this area. During this journey, it became obvious that UPM (UPE) plays a fundamental role to generate any URSM (URSE) or URSM-IM (URSE-IM).  In fact, if ${S=\{a_1,a_2, \dots a_n\}}$ is a URSM (URSE) or URSM-IM (URSE-IM), then the generating polynomial ${P_S(z)=(z-a_1)(z-a_2) \cdots(z-a_n)}$ has to be a UPM (UPE). Hence, any polynomial that is not  UPM (UPE) can never generate  URSM (URSE) and URSM-IM (URSE-IM). 
As a natural consequence, the study of UPM's (UPE's) became inevitable in order to characterize URSM's (URSE's) and URSM-IM's (URSE-IM's).
\vspace{0.1in}\par In 2000, Fujimoto \cite{fuji-2000} first made a breakthrough in this direction. He provided a characterization of UPM out of nowhere. But his characterization was limited to a special kind of polynomials called critically injective polynomials \cite{ban12}. So, before proceeding further we recall the definition of the same. 
\begin{defi}\cite{ban12,fuji-2000}
	Let ${P(z)}$ be a polynomial such that ${P'(z)}$ has ${k}$ distinct zeros, say ${d_1,d_2, \cdots,d_k}$.  Then ${P(z)}$ is said to be a critically injective polynomial or CIP in brief, 
	if ${P(d_i)\neq P(d_j)}$ for ${i\neq j}$, where ${i,j\in \{1,2, \cdots,k\} }$. The number ${k}$ is called the  derivative index of ${P(z)}$.\par A polynomial, which is not CIP is called a non critically injective polynomial or NCIP in brief. 
\end{defi}
Following is the characterization of UPM by Fujimoto.
\begin{theoA}\cite{fuji-03}
	 Let $P(z)$ be a CIP with  no multiple zeros. Let $P'(z)$  has ${k}$ distinct zeros $d_1, d_2, \ldots, d_k$ with multiplicities $q_1, q_2, \ldots, q_k$ respectively. Then $P(z)$ is UPM in ${\mathbb{C}}$
      if and only if
\begin{eqnarray}\label{bhn2}
	\sum_{1 \leq \ell<m \leq k} q_{\ell} q_m&>&\sum_{\ell=1}^k q_{\ell} .
\end{eqnarray}
\end{theoA}
For any polynomial with ${k\geq  4}$, the condition \eqref{bhn2} is always satisfied. For ${k=3}$, \eqref{bhn2} holds if only if ${\max (q_1,q_2,q_3)\geq  2 }$, and for ${k=2}$, \eqref{bhn2} holds if only if  ${\min(q_1,q_2)\geq  2}$ with ${q_1+q_2\geq  5}$.
\vspace{0.1in}\par Later in 2002, Wang \cite{taipei} added the  ${p}$-adic version of {\em Theorem A} to the literature by the following theorem. 
\begin{theoB}
		 Let $P(z)$ be a CIP and $P'(z)$  has ${k}$ distinct zeros $d_1, d_2, \ldots, d_k$ with multiplicities $q_1, q_2, \ldots, q_k$ respectively. Then $P(z)$ is UPM in ${\mathbb{K}}$
          if and only if  \begin{itemize}
		 	\item [(i)] ${k\geq  3}$; or  \item[(ii)] ${k=2}$ with ${\min (q_1,q_2)\geq  2 }$.
		 \end{itemize}
\end{theoB}
 \begin{rem}\label{remm1}
    Note that any polynomial with ${k=1}$, can never be a UPM.
    \end{rem}
    \begin{rem}\label{remm2}
    	Any polynomial with ${k=2}$ is necessarily CIP. Because, if ${P(z)}$ is an NCIP with ${k=2}$, then ${P'(z)}$ is of the form ${P'(z)=c_1(z-a)^p(z-b)^q}$ with ${P(a)=P(b)}$,  for some non-zero constant  ${c_1}$. Therefore, ${P(z)-P(a)=c_2(z-a)^{p+1}(z-b)^{q+1}}$, for some non-zero constant ${c_2}$. That is, ${P(z)}$ is a polynomial of degree ${p+q+2}$, which is impossible as degree of ${P(z)}$ is ${p+q+1}$. Hence ${P(z)}$ can not be NCIP.
    \end{rem} 
    \begin{rem} \label{remm21}
In view of {\em Remark \ref{remm1}-\ref{remm2}}, for any NCIP, we must have $k\geq 3$.
    \end{rem}
From {\em Definition 1.1}, it is obvious that any polynomial is either CIP or NCIP,  and {\em Theorem A-B} are applicable only for  CIP's. Hence the characterization of UPM for NCIP's becomes  a natural quest to the research community. 
\vspace{0.1in}\par After {\em Theorem A-B}, during the last two decades, several  results related to this problem have been obtained by many mathematicians \cite{An-Complex Variable, An-11, ban12, Mallick-cmft, mallick-filo, mallick-tbi}. 
But the problem \begin{center}
{\em`` a general characterization of NCIP to be UPM "}
\end{center}
 still remains open in the literature.
\par In this paper, we explore this open part of characterization. At the same time, we frame our main results in  a way that both CIP and NCIP can be characterized through these results. Since for any NCIP we have $k\geq 3$, so we prove our results up to the same limit of $k$. 
 Furthermore, in this characterization we take both the complex and {${p}$-adic} cases into account. 
\vspace{0.1in}\par In section-4, as applications of our main results, we exhibit a handsome number of examples of NCIP that are UPM.  At the same time, we also verify the existing examples of NCIP that are UPM. 
\par Moreover, applying a theorem of \cite{Ripan} along with the help of our main results, we shall obtain a URSM and a URSM-IM  generated by an NCIP with the least cardinality achieved so far. Till date, the least cardinality of a URSM generated by an NCIP \cite{mallick-filo, mallick-tbi} is $17$ and for URSM-IM it can be found to be 22 \cite{mallick-filo}.  Here we obtain these values as $14$ and $19$ respectively, which improve the existing results significantly.   In fact, the cardinality obtained for  URSM-IM  is the best possible with respect to  all the available formulas of URSM-IM \cite{An-11, ban12, Mallick-cmft, mallick-filo}.
\vspace{0.1in}\par Another thing we would like to mention is that a parallel research on the least degree UPM (UPE) also occurred several times in the literature  during the last four decades. Since, we provide a general characterizaton of UPM (UPE) icluding CIP and NCIP both, so on this occasion we also provide some results on the least degree of the same in both complex and {${p}$-adic} cases. To this end, we define the following  notations.
\begin{eqnarray}
	\nonumber  \mu_{I}&=& \text{ least degree of UPM for CIP, }\qquad \epsilon_I\,\,\,=\,\,\,\text{ least degree of UPE for CIP, }\\\nonumber  \mu_N&=&\text{ least degree of UPM for NCIP, }\quad \epsilon_N\,\,\,=\,\,\,\text{ least degree of UPE for NCIP. }
\end{eqnarray}
Below we briefly discuss about the lower bounds of ${ \mu  _I}$, ${ \mu  _N}$, ${ \epsilon_I}$ and ${ \epsilon_N}$ in complex and ${p}$-adic cases separately. 
\par\vspace{0.1in} {\textbf{\underline{\Large Complex case:}}} In 1995, Li-Yang \cite{ccy-kodai95} proved that any polynomial of degree 2 or 3 is not UPE in ${\mathcal{A}( \mathbb{C})}$. For polynomials of degree 4, they proved the following result.
\begin{theoC}
	Let $P(z)=z^4+a_3 z^3+a_2 z^2+a_1 z+a_0$. Then
	\begin{itemize}
		\item [(i)] $P$ is not a UPM in ${\mathbb{C}}$.
          \item[(ii)]  P is a UPE in ${\mathbb{C}}$
           if and only if
		$\dfrac{a_3^3}{8}-\dfrac{a_2 a_3}{2}+a_1 \neq 0 .$
	\end{itemize}
\end{theoC}
 {In this paper (see {\em Theorem \ref{tt7}} ), we prove that  any polynomial that satisfies the conditions} { of {\em Theorem C}, is necessarily a CIP. Thus we conclude ${\epsilon_I=4}$.} 
 Again, {\em Theorem A} gives ${\mu_{I}=5}$ and it is achieved when ${k=3}$. Moreover, in 2020, S. Mallick \cite{mallick-filo}  gave an example of UPM for NCIP of degree ${7}$ (see Example 4.3 of \cite{mallick-filo}) and an example of UPE for NCIP of degree ${5}$ (see Example 4.4 of \cite{mallick-filo}). Therefore, it is clear that ${ \mu_N\leq  7}$ and ${ \epsilon_N\leq  5}$. So, to find the exact lower bounds of ${ \mu  _N}$ and ${ \epsilon_I}$, we put forward the following questions. 
\begin{ques}\label{question1}
	Does there exist any UPM for NCIP  in $ \mathbb{C}$ of degree lesser equal to ${6}$?
\end{ques}
\begin{ques}\label{question2}
	Does there exist any UPE for NCIP  in $ \mathbb{C}$  of degree lesser equal to ${4}$?
\end{ques}
In this paper, we provide  the answer of {\em Question \ref{question1}} in affirmative. Not only that, for a good reason, we propose a conjecture that $\mu_N=6$. However, we find the answer of {\em Question \ref{question2}} as negative, and thereby we establish $\epsilon_N=5$.
\par\vspace{0.1in} {\textbf{\underline{\Large ${p}$-Adic case:}}} From {\em Theorem B}, it is clear that ${ \mu  _I=4}$ and it is achieved when ${k=3}$. In \cite{IMathBoutabaa}, authors proved that the least cardinality of unique range set is ${ 3}$. Thus, we conclude ${ \epsilon_I\leq  3}$. In this paper, we show that ${ \epsilon_N, \mu  _N\leq  6}$.
\vspace{0.1in}\par 
To prove the main results we have initiated a completely new approach. So  at first we have to formulate some basic concepts, which will play a central role to our investigations and main results. For that, we move to the next section. 
\section{{\Large { Basic Set-up }}}
Let ${P(z)}$ be any polynomial of degree ${n}$ in $\mathbb{L}$ having simple zeros only. Let \begin{eqnarray}\label{p'}
	P'(z)=a (z-d_1)^{q_1}(z-d_2)^{q_2} \dots(z-d_k)^{q_k},
\end{eqnarray}
where ${a\neq 0}$ is constant and ${q_i(\geq  1)\in\mathbb{N}}$.
Let ${S=\{z:P'(z)=0\}}$.  By ${|S|}$ we mean the cardinality of ${S}$. Observe that we can always express the set ${S}$ as \begin{eqnarray}\label{f33}
	S=S_1\cup S_2\cup \dots\cup S_s,
\end{eqnarray}
such that \begin{itemize}
	\item [(i)] ${S_i\cap S_j}= \phi $ for ${i\neq j}$;
	\item [(ii)] ${|S_1|\leq  |S_2|\leq   \dots\leq  |S_{s-1}|\leq  |S_s|}$; \item [(iii)] for each ${d_i,d_j\in S_\lambda }$ with ${i\neq j}$, ${P(d_i)\neq P(d_j)}$; \item [(iv)] for each ${d_j\in S_\lambda }$, there is one and only one element ${d_i\in S_l}$ for each ${l\in \{\lambda+1,\lambda+2, \dots, s\}}$, such that ${P(d_j)=P(d_i)}$. 
\end{itemize}
In another words, the set ${S}$ is the union of distinct disjoint set(s) ${S_i}$ such that ${P(z)}$ is injective over each set ${S_i}$ but not injective over ${S_i\cup S_j}\cup  \dots $ for any ${i,j, \dots }$. Obviously, ${P(z)}$ is injective over ${S}$ if and only if ${S_j= \phi }$ for ${j=1,2, \dots, S_{s-1}}$. 
\par\vspace{0.1in} Let
\begin{eqnarray}
	\nonumber S_1&_=&\{d_1,d_2, \dots,d_{r_{1}} \};\\\nonumber S_2&=&\{d_{r_{1}+1},d_{r_{1}+2}, \dots,d_{r_1+r_2}\};\\\nonumber \dots & \dots & \dots \quad \dots \\\nonumber S_{s-1}&=& \{d_{r_{1}+r_2+ \dots+r_{s-2}+1},d_{r_1+r_2+ \dots+r_{s-2}+2}, \dots, d_{r_1+r_2+ \dots+r_{s-2}+r_{s-1}}    \}; \\\nonumber S_{s}&=& \{d_{r_{1}+r_2+ \dots+r_{s-1}+1},d_{r_1+r_2+ \dots+r_{s-1}+2}, \dots, d_{r_1+r_2+ \dots+r_{s-1}+r_{s-1}}, \dots,d_{r_1+ \dots+r_k}    \}.
\end{eqnarray}
\par \vspace{0.1in} Take ${d_1}$ form ${S_1}$. Now, (iv) indicates that there is one and only one element ${d_j}$ in ${S_2}$ such that ${P(d_1)=P(d_j)}$. Without loss of generality, let us take ${d_j=d_{r_1+1}}$. Thus, we have paired ${d_1}$ with ${d_{r_1+1}}$ in the sense that $$P(d_1)=P(d_{r_1+1}).$$
\par Similarly, $$P(d_{r_1+1})=P(d_{r_1+r_2+1})= \dots=P(d_{r_1+ \dots+r_{s-1}+1}).$$ 
\includegraphics[width=1\linewidth]{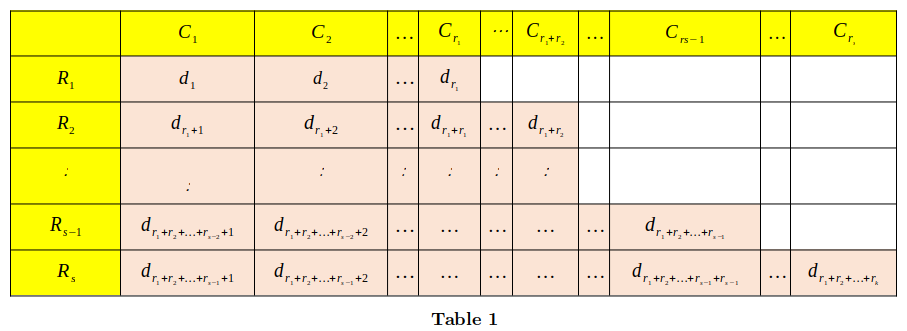}
We have arranged the elements of these set(s) in {\em Table 1} in such a way that each row ${R_l}$ denotes the set ${S_l}$. Thus for any ${d_i, d_j}$ in column ${C_l}$, we have ${P(d_i)=P(d_j)}$, and for any ${d_i,d_j}$ in ${R_l}$, we have ${P(d_i)\neq P(d_j)}$ for ${i\neq j}$.\\
\includegraphics[width=1\linewidth]{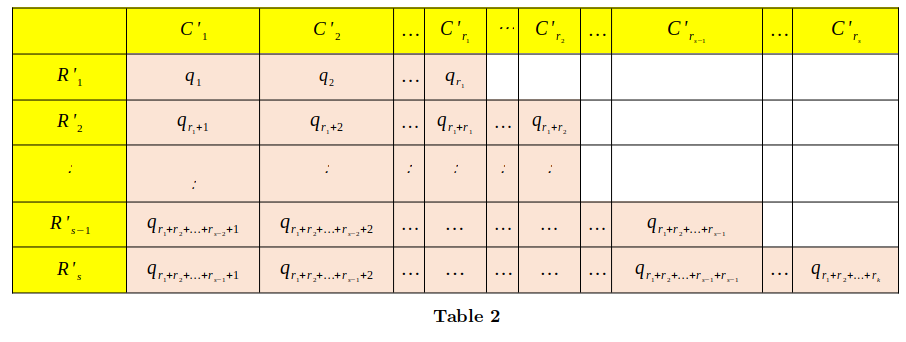} {\em Table 2} indicates the corresponding multiplicity ${q_j}$ of ${d_j}$ in ${P'(z)}$. \par \vspace{0.1in} Let \begin{eqnarray}
	\nonumber A_1&=&\{q_l\in C_1': q_l\in R'_i  \text{ but }q_l\not\in R'_j  \text{ with  }i\neq j \}.
\end{eqnarray}
Therefore, ${A_1}$ is the set of all ${q\in C'_1  \text{ such that  }q  \text{ is not repeated in  }C_1'}$. Let the set ${B_1}$ be defined corresponding to the set ${A_1}$ as:\begin{eqnarray}
	\nonumber B_1&=&\{d_l\in C_1: q_l\in A_1\}.
\end{eqnarray}
Here, each ${d_l}$ in ${B_1}$ corresponds to ${q_l}$ in ${A_1}$. Similarly, we can define  ${A_2,A_3, \dots,A_{r_{s-1}},A_{r_s}}$; and ${B_2,B_3, \dots,B_{r_{s-1}},B_{r_s}}$. Let \begin{eqnarray}\label{f34}
	A&=&\bigcup_{i=1}^{r_{s}}A_i.
\end{eqnarray} 
\begin{eqnarray}\label{f35}
	B&=&\bigcup_{i=1}^{r_{s}}B_i.
\end{eqnarray} 
\\Let \begin{eqnarray}
	\nonumber A_{i}(H_1)&=&\{q\in A_i: q>q'  \text{ for all }q'\in C_i'\setminus A_i\}.
\end{eqnarray}
 Let \begin{eqnarray}
 	\nonumber \left|A_i(H_1)\right|&=&n_i.
 \end{eqnarray}
  Since all elements in ${A_i(H_1)}$ are distinct, we can rearrange the elements of ${A_i(H_1)}$ in decreasing order as  $$q_{i1}>q_{i2}>, \dots,>q_{in_i}.$$ Let\begin{eqnarray}
  	\nonumber A_i(H_2)&=&\{q_{ij}\in A_i(H_1):q_{i1}+1<2(q_{ij}+1); 1\leq  j\leq  n_i\}.
  \end{eqnarray}
   Observe that ${A_i(H_2)\subseteq A_i(H_1)}$ and \begin{eqnarray}
   	\nonumber  \max \{q:q\in A_i(H_1)\}= q_{i1}\in  A_{i}(H_2).
   \end{eqnarray}
   Therefore, ${A_{i}(H_1)\neq  \phi \Longleftrightarrow A_i(H_2)\neq  \phi }$. Let ${B_i(H_2)}$ corresponds to the sets ${A_i(H_2)}$. Let \begin{eqnarray}\label{u1}
   	t=\left|\left\{{i:B_i(H_2)\neq  \phi }\right\}\right|\qquad  \text{ and } \qquad t'= \left| \bigcup\limits_{i}B_i(H_2)\right|.
   \end{eqnarray}
   Obviously, ${t'\geq  t}$.
   \begin{rem}\label{remm3}
   	Observe that ${P(z)}$ is CIP if and only if ${t=t'=\text{{Derivative index of ${P(z)}$}}}$.
   \end{rem}
   
\section{\Large Main Results}
\subsection{ \bf Results on the  characterization of uniqueness polynomials }
\begin{theo}\label{t1}
	For a polynomial ${P(z)}$, if ${t\geq  1}$ and ${t'\geq  5}$ ${(t'\geq 4)}$, then ${P(z)}$ is UPM (UPE) in $\mathbb{K}$.
\end{theo}
\begin{theo}\label{t2}
	For a polynomial ${P(z)}$, if ${t\geq  1}$ and ${t'\geq  6}$ ${(t'\geq 5)}$, then ${P(z)}$ is UPM (UPE) in $\mathbb{C}$.
\end{theo}
\begin{theo}\label{t3}
	For a polynomial ${P(z)}$, if ${t\geq  3}$ and ${t'\geq  3}$, then ${P(z)}$ is UPM in $\mathbb{K}$.
\end{theo}
\begin{rem}
	In view of {\em Remark \ref{remm3}}, for {any CIP with} derivative index $k\geq 3$, {Theorem \ref{t3}} { coincides with {\em Theorem B}.} 
      For $k=2$ we do not proceed further as we already mentioned earlier {(see para 3 after Remark \ref{remm21}).} 
\end{rem}
\begin{theo}\label{t4}
	For a polynomial ${P(z)}$, if ${t\geq  3}$ and ${t'\geq  4}$, then ${P(z)}$ is UPM in $\mathbb{C}$.
\end{theo}
\begin{theo}\label{t7}
Let  ${P(z)}$ be a  polynomial with
 ${t\geq  3}$; i.e., at least three ${B_i(H_2)}$'s are non empty; say ${B_1(H_2), B_2(H_2), B_3(H_2)}$ are non empty. Let ${q_i= \max  \{q:q\in C_i'\}}$ and ${d_i(\in B_i(H_2))}$ corresponds to ${q_i}$, for ${i=1,2,3.}$ {I.e.,} \begin{eqnarray}
		\nonumber P'(z)&=& (z-d_1)^{q_1}(z-d_2)^{q_2}(z-d_3)^{q_3}Q(z),
	\end{eqnarray}
	where ${Q(z)}$ is a polynomial of degree greater equal to zero. Further,  ${ (\text{say }q_1=)\max (q_1,q_2,q_3)\geq  2}$. 
	Then, ${P(z)}$ is UPM in $\mathbb{C}$, 
     whenever  \begin{itemize}
		 \item[(i)] ${\min(q_1,q_2,q_3)\geq  2}$; or
		\item [(ii)] ${P(z)}$ is CIP; or  \item[(iii)] ${P(z)}$ is NCIP with ${t'\geq  4}$  \item[(iii)] ${P(z)}$ is NCIP with  ${t'=3}$ and \begin{eqnarray}
			\nonumber \dfrac{Q'(d_j)}{Q(d_i)}&\neq & \dfrac{d_i(1+q_1)+d_1(1+q_i)-(2+q_1+q_i)d_j}{(d_j-d_1)(d_j-d_i)}, 
		\end{eqnarray}
		 for ${i\neq j}$ and ${i,j\in \{2,3\}}$.
	\end{itemize}
\end{theo}
\begin{rem}
	In view of {\em Remark \ref{remm3}}, for {CIP's} with derivative index ${k}$, {{\em Theorem \ref{t4} and {\em Theorem \ref{t7}}}} coincide with {\em Theorem A,} {whenever} ${k\geq  4}$ and {${k=3}$, respectively.}
\end{rem}
Observe that in {\em Theorem \ref{t7}}, we have assumed a condition, ${ \max  (q_1,q_2,q_3)\geq  2}$. {\em Theorem A} says that this condition is necessary for a CIP to be UPM whenever ${t=3}$. However, for NCIP this is not true. For NCIP with ${t\geq  3}$ and ${ \max  (q_1,q_2,q_3)=1}$, we can have UPM as follows.
\begin{theo}\label{t6}
 Let ${P(z)}$ be an NCIP of degree ${n(\geq  6)}$ with ${t=t'=3}$; i.e., exactly three ${B_i(H_2)}$'s are non empty, and ${|B_i(H_2)|=1}$; say ${B_1(H_2)=\{d_1\}}$, ${B_2(H_2)=\{d_2\}}$ and ${B_3(H_2)=\{d_3\}}$. Let \begin{eqnarray}
		\nonumber P'(z)&=&(z-d_1)(z-d_2)(z-d_3)Q(z),
	\end{eqnarray}
	where ${Q(z)}$ is a polynomial of degree ${\geq  2}$. Furthermore, assume 
	 \begin{itemize}
	 	 \item[(a)] ${Q'(d_j)\neq 0}$ for ${j\in \{1,2,3\} }$;
		\item [(b)] for distinct ${i,j,k\in \{1,2,3\}}$, \begin{eqnarray}
			\nonumber \dfrac{Q'(d_j)}{Q(d_i)}&\neq & \dfrac{2d_i+2d_k-4d_j}{(d_j-d_k)(d_j-d_i)}, 
		\end{eqnarray}
		 \item[(c)] for any ${\xi(\neq d_k)\in \{z:P(z)-P(d_k)=0\}}$ \begin{eqnarray}\label{xc1}
		 	\dfrac{Q(\xi )}{Q(d_k)}\neq \dfrac{(d_k-d_i)^2(d_k-d_j)^2}{(\xi-d_i)^2(\xi-d_j)^2}
		 \end{eqnarray}
		 for any ${i,j,k\in \{1,2,3\}}$.
	\end{itemize}
	Then, ${P(z)}$ is UPM in $\mathbb{C}$,
     whenever \begin{itemize}
	 	\item [(i)]  ${d_k\neq 2d_i-d_j}$; or  \item[(ii)] ${d_k=2d_i-d_j}$ with ${\dfrac{Q'(d_j)}{Q(d_j)}\neq \dfrac{3}{(d_k-d_i)}}$ for any ${i,j,k\in \{1,2,3\}}$.
	 \end{itemize}
\end{theo}
\vspace{0.2in}\par 
 Further note that, {\em Theorem \ref{t1}-\ref{t6}} provide UPM for NCIP's with the derivative index $k\geq 4$. Since  for any NCIP we have $k\geq 3$, 
 hence the following question becomes inevitable.
\begin{ques}\label{q1}
Does there exist any UPM for NCIP with $k=3$?
\end{ques}
In the next theorem, we shall provide the answer of this question in affirmative.   
\begin{theo}\label{themm9}
		Let ${P(z)}$ be an NCIP 
         with ${t=2}$ and  ${t'=3}$. Let 
		 \begin{eqnarray}
		\nonumber P'(z)&=&(z-d_1)^{q_1}(z-d_2)^{2}(z-d_3)^{q_3},
	\end{eqnarray}
	with ${B_1(H_2)=\{d_1,d_3\}}$, ${B_2(H_2)=\{d_2\}}$. Furthermore, if ${q_1\geq  6}$, then
	\begin{enumerate}
	\item[(i)]  ${P(z)}$ is UPM in $\mathbb{C}$,
       when $
		 \dfrac{q_1-1}{2}<q_3<\dfrac{q_1-2}{2}+\dfrac{\sqrt{q_1^2-4q_1-4}}{2};
		$
		\item[(ii)] ${P(z)}$ is UPM in $\mathbb{K}$,
          when $
				 \dfrac{q_1-1}{2}<q_3\leq\dfrac{q_1-2}{2}+\dfrac{\sqrt{q_1^2-4q_1-4}}{2}.
				$
	\end{enumerate} 
\end{theo}
\begin{rem}
Therefore for every value of the derivative index $k$, we can construct UPM's from NCIP's using {\em Theorem \ref{t1}}-{\em Theorem \ref{themm9}} which was one of the  prime interests of the paper.
\end{rem}
 
\subsection{\bf Results on uniqueness polynomials with least degree}
\subsubsection{\underline{\bf Complex case :}}
In the following result, we prove that the conditions in {\em Theorem C} can only be satisfied by CIP's.
\begin{theo}\label{tt7}
	A polynomial of degree $4$ with derivative index greater equal to $2$ is UPE in $\mathbb{C}$
     if and only if it is CIP.
\end{theo}
 \begin{rem}\label{remmm3}
 		An immediate conclusion of {\em Theorem \ref{tt7}} gives ${ \epsilon_N>4}$. Again, S. Mallick in \cite{mallick-filo} proved that  $P(z)=z^n+2 z^{n-1}+z^{n-2}+c$, where $n(\geq 5)$ is odd and $c \in \mathbb{C}$ be such that $P(z)$ has no multiple zero,  is a UPE for NCIP in ${ \mathbb{C}}$. Thus, we conclude ${  \epsilon_N=5}$.
 \end{rem} On the other hand, {\em Example \ref{exm8}} which is an application of {\em Theorem \ref{t4}}, reveals that ${ \mu  _N\leq  6}$ and hence the answer of {\em Question \ref{question1}} is positive. Furthermore, remembering the conclusions of {\em Remark \ref{remmm3}}, we infer that ${6}$ is the least possible degree of  UPM for NCIP; i.e., ${ \mu  _N}$ is exactly equal to ${6}$. However, we propose this fact as a conjecture.
\par\vspace{0.1in}\noindent{\textbf{CONJECTURE:}} The least degree of UPM for NCIP in $ \mathbb{C}$ is ${6}$.
\par\vspace{0.1in}\noindent In order to prove this Conjecture, one has to show that there is no UPM for NCIP of degree ${ 5}$. In this respect, we prove the  following result that sheds some light on this conjecture.
\begin{theo}\label{tt8}
Let ${P(z)=z^5+az^4+bz^3+c}$ be NCIP with ${8a^2\neq 5b}$. Then ${P(z)}$ is not UPM in $\mathbb{C}$. 
\end{theo}
\subsubsection{\underline{\bf ${p}$-Adic case :}} An application of {\em Theorem \ref{t3}} is {\em Example \ref{exam3.2}}, which shows that  ${ \epsilon_N, \mu  _N\leq  6}$.
\par\vspace{0.1in} Following Table gives a quick comparison between the  existing  and current lower bounds of $\epsilon_I,\epsilon_N,\mu_I$ and $\mu_N$.
\\\includegraphics[width=1.01\linewidth, height=0.38\textheight]{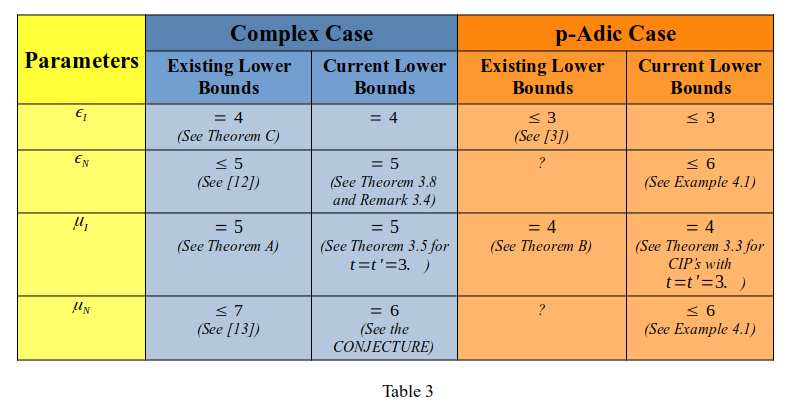}
	\section{\Large Examples}
	\begin{exm}[Application of Theorem \ref{t3}]\label{exam3.2}
		Consider the polynomial \begin{eqnarray}
			\nonumber P(z)&=&\frac{1}{6}z^6-\frac{186 }{53}z^5+\frac{1565 }{53}z^4-\frac{6630 }{53}z^3+\frac{28967 }{106}z^2-\frac{14460}{53}z+1.
		\end{eqnarray}
		Here \begin{eqnarray}
			\nonumber P'(z)&=&\frac{1}{53} (z-5) (z-4) (z-3) (z-1) (53 z-241).
		\end{eqnarray}
		So, \begin{eqnarray}
			\nonumber S=\{z:P'(z)=0\}=\left\{{1,\dfrac{241}{53},3,4,5}\right\}.
		\end{eqnarray}
		Now  \begin{itemize}
			\item ${P(1)=-\dfrac{15497}{159}}$\item ${P\left({\dfrac{241}{53}}\right)=-\dfrac{5030097474637}{66493083387}}$\item ${P(3)=-\dfrac{3979}{53}}$\item ${P(4)=P(5)=-\dfrac{12041}{159}}$.
		\end{itemize}
		\begin{center}
			{\em Table 1}\qquad\qquad\qquad\qquad\qquad\qquad\qquad\qquad{\em Table 2}\\\begin{tabular}{|c|c|c|c|}
				\hline
				${5}$&. & . &.  \\
				\hline
				${4}$&${\dfrac{241}{53}}$  &${3}$  &${1}$  \\
				\hline
			\end{tabular}\qquad\qquad\qquad\qquad\qquad\qquad\begin{tabular}{|c|c|c|c|}
				\hline
				${1}$	&.  & . &.  \\
				\hline
				${1}$&${1}$  &${1}$  &${1}$  \\
				\hline
			\end{tabular}
		\end{center}
		Clearly, ${P(z)}$ is NCIP with ${t=t'=3}$. Therefore, applying {\em Theorem \ref{t3}}, we see that ${P(z)}$ is UPM in $\mathbb{K}$.
	\end{exm} 
	\begin{rem}
		 {\em Example \ref{exam3.2}} clearly
		 shows that ${ \epsilon_N, \mu  _N\leq  6}$ for {${p}$-adic} case, as we claimed before. 
	\end{rem}
	\begin{exm}[Application of Theorem \ref{t3} and Theorem \ref{t4}]
		Consider the polynomial \begin{eqnarray}
			\nonumber P(z)&=&\frac{1}{7}z^7-\frac{23105 }{8379}z^6+\frac{19279 }{931}z^5-\frac{4285 }{57}z^4+\frac{122428 }{931}z^3-\frac{253880 }{2793}z^2+1.
		\end{eqnarray}
		Here \begin{eqnarray}
			\nonumber P'(z)&=&\dfrac{1}{2793}z(z-5) (z-4) (z-2) (z-1)(2793 z-12694).
		\end{eqnarray}
		So, \begin{eqnarray}
			\nonumber S=\left\{{0,1,2,\frac{12694}{2793},4,5}\right\}.
		\end{eqnarray}
		Now  \begin{itemize}
			\item ${P(0)=1}$\item ${P(1)=-\dfrac{129701}{8379}}$\item ${P(2)=-\dfrac{10691}{1197}}$\item ${P\left({\dfrac{12694}{2793}}\right)=-\dfrac{858908850511840736130799715}{27842988283701433932953997}}$.\item ${P(4)=P(5)=-\dfrac{263621}{8379}}$.
		\end{itemize}
		\begin{center}
			{\em Table 1}\qquad\qquad\qquad\qquad\qquad\qquad\qquad\qquad{\em Table 2}\\\begin{tabular}{|c|c|c|c|c|}
				\hline
				${5}$&. & . &. & \\
				\hline
				${4}$&$0$  &${1}$  &${2}$ &${\dfrac{12694}{2793}}$ \\
				\hline
			\end{tabular}\qquad\qquad\qquad\qquad\qquad\qquad\begin{tabular}{|c|c|c|c|c|}
				\hline
				${1}$	&.  & . &. & .\\
				\hline
				${1}$&${1}$  &${1}$  &${1}$ &${1}$ \\
				\hline
			\end{tabular}
		\end{center}
		Therefore, ${P(z)}$ is NCIP with ${t=t'=4}$. Thus, applying {\em Theorem \ref{t4}} (or {\em Theorem \ref{t3}}), we get ${P(z)}$ is UPM in $\mathbb{C}$ (and in  $\mathbb{K}$).
	\end{exm}
	
	\begin{exm}[Application of Theorem \ref{t3} and Theorem \ref{t7}]
		Consider the polynomial \begin{eqnarray}
			\nonumber P(z)&=&\frac{z^7}{7}-\frac{4071 z^6}{1316}+\frac{1277 z^5}{47}-\frac{81325 z^4}{658}+\frac{101342 z^3}{329}-\frac{540647 z^2}{1316}+\frac{90030 z}{329}+1.
		\end{eqnarray}
		Here \begin{eqnarray}
			\nonumber P'(z)&=&\frac{1}{658} (z-5) (z-4) (z-3) (z-1)^2 (658 z-3001).
		\end{eqnarray}
		So, \begin{eqnarray}
			\nonumber S=\left\{{1,\dfrac{3001}{658},3,4,5}\right\}.
		\end{eqnarray}
		Now \begin{itemize}
			\item ${P(1)=\dfrac{23845}{329}}$\item ${{P\left({\dfrac{3001}{658}}\right)=\dfrac{66183058741702202837617}{747668856695865052928}}}$\item ${P(3)=\dfrac{4223}{47}}$\item ${P(4)=P(5)=\dfrac{4147}{47}}$.
		\end{itemize}
		\begin{center}
			{\em Table 1}\qquad\qquad\qquad\qquad\qquad\qquad\qquad\qquad{\em Table 2}\\\begin{tabular}{|c|c|c|c|}
				\hline
				${5}$&. & . &.  \\
				\hline
				${4}$&${1}$  &${\dfrac{3001}{658}}$  &${3}$  \\
				\hline
			\end{tabular}\qquad\qquad\qquad\qquad\qquad\qquad\begin{tabular}{|c|c|c|c|}
				\hline
				${1}$	&.  & . &.  \\
				\hline
				${1}$&${2}$  &${1}$  &${1}$  \\
				\hline
			\end{tabular}
		\end{center}
		Thus, ${P(z)}$ is NCIP with ${t=t'=3}$. Therefore, applying {\em Theorem \ref{t7}} (or {\em Theorem \ref{t3}}), we conclude ${P(z)}$ is UPM in $\mathbb{C}$ (and in $\mathbb{K}$).
	\end{exm}
	\begin{exm}[Application of Theorem \ref{t3}, Theorem \ref{t4} and  Theorem \ref{t7}]\label{exm8}
		Consider the polynomial \begin{eqnarray}
			\nonumber P(z)&=&\frac{z^6}{6}-\left(\frac{6}{5}+\frac{2 i}{5}\right) z^5+\left(\frac{5}{2}+3 i\right) z^4-\frac{22 i }{3}z^3-\left(\frac{11}{2}-6 i\right) z^2+6 z,
		\end{eqnarray}
		where \begin{eqnarray}
			\nonumber P'(z)&=&(z-i)^2(z-1)(z-2)(z-3).
		\end{eqnarray}
		Here \begin{eqnarray}
			\nonumber S&=&\{i,1,2,3\},
		\end{eqnarray}
		with  \begin{itemize}
			\item ${P(i)=P(3)=\dfrac{9}{10}+\dfrac{9 i}{5}}$;\item ${P(1)=\dfrac{59}{30}+\dfrac{19 i}{15}}$;\item ${P(2)=\dfrac{34}{15}+\dfrac{8 i}{15}}$.
		\end{itemize}
		Clearly, ${P(z)}$ is NCIP having simple zeros only. Here  \begin{center}
			{\em Table 1}\,\,\,\,\qquad\qquad\qquad\quad {\em Table 2}\\\begin{tabular}{|c|c|c|}
				\hline
				${3}$&. &.  \\
				\hline
				${i}$& ${2}$ &${1}$  \\
				\hline
			\end{tabular}
			\qquad\qquad\qquad	\begin{tabular}{|c|c|c|}
				\hline
				${1}$&.  &.  \\
				\hline
				${2}$&${1}$  &${1}$  \\
				\hline
			\end{tabular}
		\end{center}
		So, ${t=3}$, ${t'=4}$. Therefore, using {\em Theorem \ref{t4}} (or {\em Theorem \ref{t3}}), we see that ${P(z)}$ is UPM in $\mathbb{C}$ (and in $\mathbb{K}$).\par Observe that  ${P(z)}$ also satisfies the conditions of {\em Theorem \ref{t7}}.
	\end{exm}
		\begin{rem}
			 {\em Example \ref{exm8}} 
			 shows that ${\mu  _N\leq  6}$ for complex case. 
		\end{rem}
			\begin{exm} [Application of Theorem \ref{t3} and Theorem \ref{t7}]\label{ex1}
			Let
			$$
			P(z)=z^{n}+a z^{n-m}+b z^{n-2 m}+c,
			$$
			where $a, b \in \mathbb{C}^{*}$ and $n, m \in \mathbb{N}$ such that $\operatorname{gcd}(m, n)=1, n>2 m+1, a^{2}=4 b$. Then
			$$
			P^{\prime}(z)=z^{n-2 m-1}\left(z^{m}+\frac{a}{2}\right)\left(n z^{m}+(n-2 m) \frac{a}{2}\right) .
			$$
			Suppose $e_{i}$ and $c_{i}$ be the distinct roots of the equation $z^{m}=-\frac{(n-2 m) a}{2 n}$ and $z^{m}=-\frac{a}{2}$ respectively for $i=1,2, \ldots, m$.
			Further suppose that $\lambda_{i}=-\left(e_{i}^{n}+a e_{i}^{n-m}+b e_{i}^{n-2 m}\right)$ for $i=1,2, \ldots, m$. Therefore $e_{i}$ 's are not zeros of $P(z)$ if $c \neq \lambda_{i}$. Also note that as $a^{2}=4 b$, so
			$$
			P(z)=z^{n-2 m}\left(z^{m}+\frac{a}{2}\right)^{2}+c=z^{n-2 m} \prod_{i=1}^{m}\left(z-c_{i}\right)^{2}+c .
			$$
			Hence $z=0, c_{i}$ are not zeros of $P(z)$ if $c \neq 0$. So clearly $P(z)$ has no multiple zero if $c \neq 0, \lambda_{i}$. Observe that $P(0)=P\left(c_{i}\right)=c$, which implies that $P(z)$ is an NCIP.
			\begin{center}
				\includegraphics[width=0.7\linewidth]{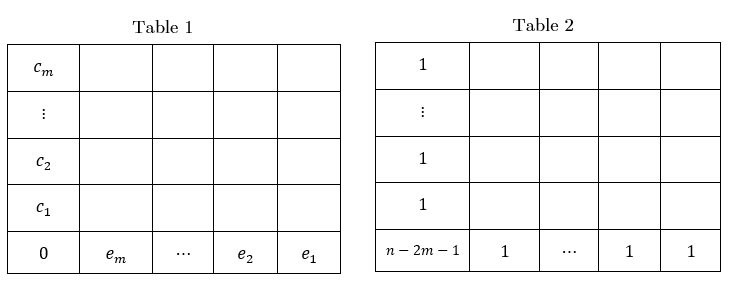}
			\end{center}
			Clearly, for ${n\geq  2m+3}$, we have ${t=m+1}$, ${t'=m+1}$. Then applying {\em Theorem \ref{t7}} (or {\em Theorem \ref{t3}}), we see that $P(z)$ is a UPM in $\mathbb{C}$ (and in $\mathbb{K}$), whenever $m\geq 2$. 
		\end{exm}
		\begin{rem}
			Note that the above example of NCIP was given by Mallick in \cite[see page 89-92]{mallick-tbi} with a long proof of almost four pages. This happened because there was no ready formula till then, to prove  this NCIP to be UPM. But now applying our theorems we have got  a tiny proof of the same just by forming the required tables.
		\end{rem}
\begin{exm}[Application of Theorem \ref{t3} and  Theorem \ref{t6}]\label{ext6}
    Consider the polynomial ${P(z)}$, where \begin{eqnarray}\label{nj1}
        P(z)&=&\dfrac{1}{6}z^6+\left(-\frac{11}{20}+\frac{1}{4} i \sqrt{\frac{19}{5}}\right)z^5+\left(-\frac{9}{16}-\frac{i \sqrt{95}}{16}\right)z^4\\\nonumber&&+\left(\frac{11}{3}-\frac{i \sqrt{95}}{3}\right)z^3+\left(-\frac{7}{2}+\frac{i \sqrt{95}}{2}\right)z^2+c,
    \end{eqnarray}
where ${c}$ is an arbitrary constant.
    Here \begin{eqnarray}\label{nj2}
        P'(z)&=&\frac{1}{4}z (z-1) (z-2) (z+2) \left(4 z+i \sqrt{95}-7\right)\\\nonumber &=& z(z-1)(z-2)Q(z),
    \end{eqnarray}
    where ${Q(z)=\dfrac{1}{4}(z+2)(4z+i\sqrt{95}-7)}$. Now one can verify that ${B_1(H_2)=\{0\}}$, ${B_2(H_2)=\{1\}}$, ${B_3(H_2)=\{2\}}$ and  \begin{eqnarray}
        \nonumber P(-2)&=&P \left(\dfrac{7}{4}-i \dfrac{\sqrt{95}}{4}\right).
    \end{eqnarray}
    So, ${P(z)}$ is NCIP. Now we see that ${P(z)}$ satisfies all the hypotheses assumed in {\em Theorem \ref{t6}}. Hence ${P(z)}$ is an NCIP, which is UPM in $\mathbb{C}$ (and in $\mathbb{K}$).
\end{exm}
\begin{exm}[Application of Theorem \ref{themm9} ]\label{exm7}
  Consider the polynomial 
  \begin{eqnarray}\label{vn1}
      P_1(z)&=&\frac{1}{2366}{z^8 (z-1)^5 \left(169 z+8 i \sqrt{35}-107\right)}+c_1,
  \end{eqnarray}
  where ${c_1}$ is any arbitrary constant. Here 
   \begin{eqnarray}
      \nonumber P_1'(z)&=&z^7(z-1)^4\left(z-\dfrac{56}{91}+\dfrac{2i\sqrt{35}}{91}\right)^2.
  \end{eqnarray}
  Clearly,  ${B_1(H_2)=\{0,1\}}$ and ${B_2(H_2)= \left\{\dfrac{56}{91}-\dfrac{2i\sqrt{35}}{91}\right\}  }$; i.e., ${t=2,t'=3.}$ Now one can easily verify that ${P_1(z)}$ satisfies all the conditions of {\em Theorem \ref{themm9}}. Hence  ${P_1(z)}$ is an NCIP, which is UPM in $\mathbb{C}$ (or in $\mathbb{K}$). \par \vspace{0.1in} Let us consider another polynomial ${P_2(z)}$, where 
  \begin{eqnarray}\label{vn2}
      P_2(z)&=&\frac{1}{3078}{ z^{11} (z-1)^7\left(162 z+i \sqrt{1463}-101\right)}+c_2,
  \end{eqnarray}
  where $c_2$ is arbitrary constant. Then
  \begin{eqnarray}
      \nonumber P_2'(z)&=&z^{10}(z-1)^6\left(z-\dfrac{209}{342}+\dfrac{i\sqrt{1463}}{342}\right)^2.
  \end{eqnarray}
  Here  ${B_1(H_2)=\{0,1\}}$, ${B_2(H_2)=\left\{\dfrac{209}{342}-\dfrac{i\sqrt{1463}}{342}\right\}}$ and ${P_2(z)}$ satisfies all the conditions of {\em Theorem \ref{themm9}}. Therefore,   ${P_2(z)}$ is  also an NCIP, which is UPM in $\mathbb{C}$ (and in $\mathbb{K}$).
\end{exm}
In order to show the  further applications of {\em Theorem \ref{themm9}}, let us quickly recall the following result.
\begin{theoD}\cite{Ripan}
    	Let $P(z)=a_nz^n+a_{n-1}z^{n-1}+\ldots+a_1z+a_0,$
    where $ a_0,a_1,\dots,a_n$ are  complex numbers with $a_n,a_0\neq 0$, $a_{i}$ being the first non-vanishing coefficient  from $a_{n-1}, a_{n-2},\ldots,a_{1}$. Let $S=\{z:P(z)=0\}.$
    Observe that $P(z)$ can be written in the form \be\label{el000}P(z)=a_{n}\prod\limits_{i=1}^{p}(z-\alpha_{i})^{m_{i}}+a_{0},\ee where $p$ denotes the   number of distinct zeros of $P(z)-a_{0}$. We also denote by $k$ the number of distinct zeros of $P^{'}(z)$. Let \begin{enumerate}
        \item[(i)] $p \geq 4$ or
        \item[(ii)] $p = 3$ and $gcd(m_i, n) = 1$ for at least one of the $m_i'$s such that $m_i \geq 2$ or
        \item[(iii)] $p = 3$ and $gcd(m_1, n)\ne 1$, where $ m_1 \geq 2,m_2=m_3=1$ and $n=\sum_{i=1}^{3}m_i \geq 5$ or
        \item[(iv)] $p=2$ and $gcd(m_i,n)=1$ for at least one of the $m_i'$s such that $n=\sum_{i=1}^{3}m_i \geq 5$ or
        \item[(v)] $p=2$ and $gcd(m_i,n)\ne 1$ for each $m_i$ such that $n \geq 2(b_1+b_2)+1$, where $b_1 = gcd(m_1,n)$ and $b_2=gcd(m_2,n)$.
    \end{enumerate}

    Suppose $f,g\in \mathcal{M}(\mathbb{C})$ be such that $f^{-1}(S)=g^{-1}(S)$.
    Then 
the following are equivalent:
\begin{itemize}
\item[(A)] $P(z)$ is a UPM;
\item[(B)]  $S$ is a URSM for $n\geq2k+7$;
\item[(C)]  $S$ is a URSM-IM for $n\geq2k+13$.

\end{itemize}
\end{theoD}
For the polynomial ${P_1(z)}$ mentioned in {\em Example \ref{exm7}}, we see that the derivative index ${k=3}$ and the degree of the polynomial is ${n=14}$. So, ${n> 2k+7}$. Also, the condition \textit{(ii)} of {\em Theorem D} is  satisfied. Thus applying {\em Theorem D}, we conclude that the zeros of the NCIP ${P_1(z)}$ forms a URSM that contains ${14}$ elements. Similarly, the zeros of ${P_2(z)}$ forms a URSM-IM that contains ${19}$ elements.
     \par \vspace{0.1in} To proceed further we need the following definitions.
	\section{\Large Definitions}
	
	
	 For standard notations and definitions on Complex and ${p}$-adic Nevanlinna calculus, we refer our reader to follow \cite{p-adBou-90,p-adkhai-88,heyman}. Though, for the convenience of the reader, we recall some  of them which  will be frequently used throughout the paper. From now on, by any meromorphic (entire) function $f$, we mean $f\in \mathcal{M}(\mathbb{L})$ $( \mathcal{A}(\mathbb{L}))$ unless otherwise stated.
	\begin{defi}
	For any non-constant meromorphic function ${h}$, by ${S(r,h)}$, we mean any quantity, which is equal to ${o(T(r,h))}$ as ${r \to \infty}$, ${r\not\in E}$, where ${E}$ is a set of positive real numbers with finite linear measure.
	\end{defi}
	\begin{defi}
		For two non-constant meromorphic functions ${f}$ and ${g}$, we say that ${f,g}$ share the value ${a}$ CM(counting multiplicities)  if ${f-a}$ and ${g-a}$ have same set of zeros, counted with multiplicities. On the other hand, if we ignore the multiplicities, we say that ${f,g}$ share the value ${a}$ IM (ignoring multiplicities).
	\end{defi}
	\begin{defi}\cite{mues-89,c.c.y. book}
	By ${N_E(r,a)}$, we denote the reduced counting function of  common zeros of ${f-a}$ and ${g-a}$ with the same multiplicities. 
	\end{defi}
	\begin{defi}\cite{mues-89,c.c.y. book}
		Let ${f,g}$ be two non-constant meromorphic functions. We say that ${f,g}$ share the value ${a}$ \textbf{``CM"}, if \begin{eqnarray}
			\nonumber  \overline{N}(r,a;f)-N_E(r,a)&=&S(r,f)
		\end{eqnarray}
		and \begin{eqnarray}
			\nonumber  \overline{N}(r,a;g)-N_E(r,a)&=&S(r,g).
		\end{eqnarray}
	\end{defi}
	Clearly, if ${f,g}$ share a value ${a}$ CM, then ${f,g}$ share ${a}$ \textbf{``CM"}, but the converse is not true.	
	\section{\Large {Lemmas}}
	Let us choose a constant ${c}$ such that ${c+d_i\neq 0}$ for all ${d_i\in B}$. Now we define \begin{eqnarray}\label{phi}
		\Phi&=&\dfrac{1}{f+c}-\dfrac{1}{g+c}.
	\end{eqnarray}
The assumption ${P(f)=P(g)}$ implies
\begin{eqnarray}\label{fri1}
    T(r,f)=T(r,g)+O(1) \quad \text{ and }\quad S(r,f)=S(r,g).
\end{eqnarray}

	\begin{lem}\label{lem1}
	Let ${d\in S}$. If ${f(z_0)=d}$ of multiplicity ${p_1(\geq 1)}$ and ${g(z_0)=d}$ of multiplicity ${p_2(\geq 1)}$, then ${P(f)=P(g)}$ implies ${p_1=p_2}$.
\end{lem}
\begin{proof}
	Let ${f(z_0)=d}$ of multiplicity ${p_1}$ and ${g(z_0)=d}$ of multiplicity ${p_2}$. Let ${q_d}$ (in {\em Table 2}) corresponds to ${d}$ (in {\em Table 1}). Now ${P(f)=P(g)}$ implies that \begin{eqnarray}\label{xss16}
		\nonumber P(f)-P(d)&=&P(g)-P(d)\\ \Rightarrow  (f-d)^{q_d+1}Q(f) &=& (g-d)^{q_d+1}Q(g),
	\end{eqnarray}
	where ${Q(z)}$ is a polynomial with ${Q(d)\neq 0}$. Clearly, \eqref{xss16} implies $p_1=p_2$.
\end{proof}
\begin{lem}\label{lem4}
	Let ${d\in B_i(H_2)}$. \begin{itemize}
		\item [(i)] If ${f(z_0)=d}$ with ${g'(z_0)\neq 0}$, then ${P(f)=P(g)}$ implies ${g(z_0)=d}$. \item [(ii)] If ${g(z_0)=d}$ with ${f'(z_0)\neq 0}$, then ${P(f)=P(g)}$ implies ${f(z_0)=d}$.
	\end{itemize}
\end{lem}
\begin{proof}
	We shall only prove (i). The proof of (ii) is similar.
	\par Without loss of generality, let us assume that ${d\in B_1(H_2)}$. Let \begin{eqnarray}
		\nonumber 	B_1(H_2)&=&\{d_{11},d_{12},\cdots, d_{1m_1} \}
	\end{eqnarray}
	and \begin{eqnarray}
		\nonumber 	A_1(H_2)&=&\{q_{11},q_{12},\cdots, q_{1m_1} \},
	\end{eqnarray}
	where ${q_{1j}}$ (in {\em Table 2}) corresponds to ${d_{1j}}$ (in {\em Table 1}).
	\par   Let ${d_1,d_2,\cdots,d_u}$ be all the zeros of ${P'(z)}$ such that ${d_j\in C_1\setminus B_1(H_2)}$ and ${q_j\in C_1'\setminus A_1(H_2)}$, for ${j=1,2,\cdots,u}$. Therefore, \begin{eqnarray}\label{cvn1}
		C_1&=&\{d_{11},d_{d_{12}},\cdots,d_{1m_1},d_1,\cdots,d_u\}.
	\end{eqnarray}
	Again, as ${P(d_{11})=P(d_{{12}})=\cdots=P(d_{1m_1})=P(d_1)=\cdots P(d_u)}$, we can write  \begin{eqnarray}\label{xs4}
		P(z)-P(d_{1j})&=& \left(\prod_{j=1}^{m_1}(z-d_{1j})^{q_{1j}+1} \right) \left(\prod_{j=1}^{u}(z-d_j)^{q_j+1}\right)Q(z),
	\end{eqnarray}
	where ${Q(z)}$ is a polynomial of degree ${\geq 0}$, and ${S\cap \{z:Q(z)=0\}=\phi}$.
	\par  Again, from the definition of sets ${A_1(H_2)}$ and ${B_1(H_2)}$, we see that \begin{eqnarray}\label{xs*}
		q_j&<&\min\{q_{11},\cdots,q_{1m_1}\}.
	\end{eqnarray}
	Now, ${P(f)-P(d_{1j})=P(g)-P(d_{1j})}$ implies
	\begin{eqnarray}\label{xs5}
		&& \left(\prod_{j=1}^{m_1}(f-d_{1j})^{q_{1j}+1} \right) \left(\prod_{j=1}^{u}(f-d_j)^{q_j+1}\right)Q(f)\\\nonumber=&&  \left(\prod_{j=1}^{m_1}(g-d_{1j})^{q_{1j}+1} \right) \left(\prod_{j=1}^{u}(g-d_j)^{q_j+1}\right)Q(g).
	\end{eqnarray}
	
	Again, differentiating ${P(f)=P(g)}$ both side, we get \begin{eqnarray}\label{xs**}
		f'P'(f)&=&g'P'(g).
	\end{eqnarray}
	\par Let ${f(z_0)=d_{11}}$ of multiplicity ${p_1}$ and ${g(z_0)=b}$ of multiplicity ${q_1}$. Let ${g'(z_0)\neq 0}$, i.e., ${q_1=1}$.
	Now from \eqref{xs**}, we see that ${g(z_0)\in S}$.
	Again, as ${S\cap \{z:Q(z)=0\}=\phi}$, we have ${Q(b)\neq 0}$. Therefore, from \eqref{xs5}, we must have either ${\prod_{j=1}^{u}(b-d_j)=0}$, or ${\prod_{j=1}^{m_1}(b-d_{1j})=0}$. 
	\par \vspace{0.1in}{\textbf{Case-1:}} Let ${\prod_{j=1}^{u}(b-d_j)=0}$. Now
     comparing the powers of $(z-z_0)$ on both  sides of \eqref{xs5}, we get \begin{eqnarray}
		\nonumber p_1(q_{11}+1)&=&q_j+1\\\nonumber &< & 1+\min\{q_{11},\cdots,q_{1m_1}\}, \quad \text{ [using \eqref{xs*}]}
	\end{eqnarray}
	which is a contradiction, as ${p_1\geq 1}$.
	\par\vspace{0.1in} {\textbf{Case-2:}} Let ${\prod_{j=1}^{m_1}(b-d_{1j})=0}$, i.e., ${b=d_{1i}}$ for some ${i\in \{1,2,\cdots,m_1\} }$. Again, from \eqref{xs5}, we see that \begin{eqnarray}\label{xs7}
		p_1(q_{11}+1)&=&q_{1i}+1, \quad\text{ for some ${i\in \{1,2,\cdots,m_1\} }$}.
	\end{eqnarray}
	Since ${q_{11}>q_{1i}}$ for ${i=2,3,\cdots,m_1}$, so \eqref{xs7} holds only if ${i=1}$ and ${p_1=1}$, i.e., ${g(z_0)=b=d_{11}}$. \par Thus in any case, we see that \begin{eqnarray}\label{xs9}
		f(z_0)=d_{11},\,\,\,\, g'(z_0)\neq 0&\Rightarrow &g(z_0)=d_{11}.
	\end{eqnarray}	
	\par\vspace{0.1in} Let ${f(z_0)=d_{12}}$ of multiplicity ${p_2}$ and ${g'(z_0)\neq 0}$. Then from \eqref{xs**}, we have ${g(z_0)\in S}$. Let ${g(z_0)= b_2 }$ of multiplicity ${1}$. Again, as ${S\cap \{z: Q(z)=0\}=\phi}$, putting ${z=z_0}$ in \eqref{xs5}, we get either ${\prod_{j=1}^{u}(b_2-d_{j})=0}$ or ${\prod_{j=1}^{m_1}(b_2-d_{1j})=0}$.
	\par  Now, if ${\prod_{j=1}^{u}(b_2-d_{j})=0}$, then  proceeding similarly like {\em Case-1}, we arrive at a contradiction. \par If ${\prod_{j=1}^{m_1}(b_2-d_{1j})=0}$, then proceeding similarly like {\em Case-2}, we see that ${b_2\not \in  \{d_{13},d_{14},\cdots d_{1u}\}}$. Therefore, we must have ${b_2\in \{d_{11},d_{12}\}}$. Let ${b_2=d_{11}}$. Now putting ${z=z_0}$ in \eqref{xs5}, we see that \begin{eqnarray}\label{xs10}
		p_2(1+q_{12})&=&1+q_{11}.
	\end{eqnarray}
	Clearly, ${p_2\neq 1}$ as ${q_{11}>q_{12}}$. So, ${p_2\geq 2}$. Hence ${1+q_{11}\geq 2(1+q_{12})}$, which is a contradiction as ${d_{11},d_{12}\in B_1(H_2)}$ and ${q_{11}=\max \{q:q\in C_1'\}}$. Thus the only possibility is ${b_2=d_{12}}$. \par Therefore,
	\begin{eqnarray}\label{xs11}
		f(z_0)=d_{12},\,\,\,\, g'(z_0)\neq 0 &\Rightarrow &g(z_0)=d_{12}.
	\end{eqnarray}
	\par Continuing in this way, for any ${d\in \{d_{11},d_{12},\cdots,d_{1m_1}\}}$, we see that ${f(z_0)=d}$, ${g'(z_0)\neq 0}$ implies ${g(z_0)=d}$.
	\par This completes the proof of the lemma.
	
\end{proof}

\begin{lem}\label{e2}
	Let ${d\in B_l(H_2)}$. \begin{itemize}
		\item [(i)] If ${f(z_0)=d}$; ${g'(z_0)=0}$; ${f'(z_0)\neq 0}$, then ${P(f)=P(g)}$ implies  ${g(z_0)\not\in S.}$
		\item [(ii)] If ${g(z_0)=d}$; ${f'(z_0)=0}$; ${g'(z_0)\neq 0}$, then ${P(f)=P(g)}$ implies ${f(z_0)\not\in S.}$
	\end{itemize}
\end{lem}
\begin{proof}
	Let ${f(z_0)=d\in B_l(H_2)}$. Now ${P(f)=P(g)}$ implies either ${g(z_0)\in C_l}$ or  ${g(z_0)\not\in S}$.\par Assume that ${g(z_0)\in C_l}$. Using {\em Lemma \ref{lem1}}, we see that ${g(z_0)\neq d}$. Let ${g(z_0)=d'(\neq d)\in C_l}.$  Again, applying {\em Lemma \ref{lem4}}, we get
	\begin{eqnarray}
		\nonumber g(z_0)=d',f'(z_0)\neq 0&\Rightarrow &f(z_0)=d',
	\end{eqnarray}
	which is a contradiction as ${f(z_0)=d(\neq d')}$. Hence ${g(z_0)\not\in S}$.
	\par Proof of (ii) is similar.
\end{proof}
\begin{lem}\label{e4}
	Let ${d\in B_i(H_2)}$. \begin{itemize}
		\item [(i)] If ${f(z_0)=d}$; ${g'(z_0)=0}$ with ${f'(z_0)=0}$ and ${g(z_0)=d'(\neq d)\in S}$; or \item [(ii)] If ${g(z_0)=d}$; ${f'(z_0)=0}$ with ${g'(z_0)=0}$ and ${f(z_0)=d'(\neq d)\in S}$, 
	\end{itemize}
	then  ${P(f)=P(g)}$ implies  \begin{itemize}
		\item [(I)] ${f''(z_0)=0}$, when ${q_{d'}>q_d}$; \item [(II)] ${g''(z_0)=0}$, when ${q_{d'}<q_d}$,
	\end{itemize}
	where ${q_d,q_{d'}}$ (in {\em Table 2}) corresponds to ${d,d'}$ (in {\em Table 1}) respectively.
\end{lem}
\begin{proof}
	We shall only prove (i). The proof of (ii) is similar.
	\par Let ${B_i(H_2)=\{d_{i1},d_{i2},\cdots,d_{im_i}\}}$ and ${A_i(H_2)=\{q_{i1},q_{i2},\cdots,q_{im_i}\}}$, where ${d_{ij}}$ (in {\em Table 1}) corresponds to ${q_{ij}}$ in {\em Table 2}.
	\par  Let ${f(z_0)=d_{ik}}$ of multiplicity ${\alpha (\geq 2) }$, for some ${k\in \{1,2,\cdots,m_i\}}$. From ${P(f)=P(g)}$, we see that ${d_{ik}, g(z_0)}$ belong to the same column ${C_i}$. Let ${g(z_0)=d'}$ of multiplicity ${\beta (\geq 2) }$, for some ${d'\in C_i}$ with ${d_{ik}\neq d'}$. Let ${q_{ik}}$ and ${q'}$ (in {\em Table 2}) be the corresponding elements of ${d_{ik}}$ and ${d'}$ (in {\em Table 1}) respectively. Obviously ${q_{ik}}\neq q'$. 
	\par Let ${d_1,d_2,\cdots,d_u}$ are all the zeros of ${P'(z)}$ such that ${d_j\in C_i\setminus B_i(H_2)}$ and ${q_j\in C_i'\setminus A_i(H_2)}$, where ${j=1,2,\cdots,u}$.  That is \begin{eqnarray}
		\nonumber C_i&=&\{d_{i1},d_{i2},\cdots,d_{im_i},d_1,\cdots,d_u\}.
	\end{eqnarray}
	Therefore, \begin{eqnarray}\label{xs12}
		P(z)-P(d_{ik})&=& \left(\prod_{j=1}^{m_i}(z-d_{ij})^{q_{ij}+1} \right) \left(\prod_{j=1}^{u}(z-d_j)^{q_j+1}\right)Q(z),
	\end{eqnarray}
	where ${Q(z)}$ is a polynomial of degree ${\geq 0}$ and ${S\cap \{z:Q(z)=0\}=\phi}$. Again, from the definition of sets ${A_i(H_2)}$ and ${B_i(H_2)}$, we see that \begin{eqnarray}\label{xss*}
		q_j&<&\min\{q_{i1},\cdots,q_{im_i}\}.
	\end{eqnarray}
	Now, ${P(f)-P(d_{ij})=P(g)-P(d_{ij})}$ implies
	\begin{eqnarray}\label{xss5}
		&& \left(\prod_{j=1}^{m_i}(f-d_{ij})^{q_{ij}+1} \right) \left(\prod_{j=1}^{u}(f-d_j)^{q_j+1}\right)Q(f)\\\nonumber=&&  \left(\prod_{j=1}^{m_i}(g-d_{ij})^{q_{ij}+1} \right) \left(\prod_{j=1}^{u}(g-d_j)^{q_j+1}\right)Q(g).
	\end{eqnarray}
	\par {\textbf{Case-1:}} Let ${q_{ik}<q'}$. We shall show that ${\alpha \geq 3}$. On the contrary,  let us assume that ${\alpha = 2}$. Now  from \eqref{xss5}, we see that \begin{eqnarray}
		\nonumber\alpha (q_{ik}+1)&=&\beta (q'+1)\\\nonumber \text{i.e., }\quad 2(q_{ik}+1) &=& \beta(q'+1)\\\nonumber &>& \beta (q_{ik}+1),
	\end{eqnarray}
	which is a contradiction as ${\beta \geq 2}$. Hence we must have ${\alpha \geq 3}$, i.e., ${f''(z_0)=0}$. This proves (I).
	\par {\textbf{Case-2:}} Let ${q_{ik}>q'}$. We shall show that ${\beta \geq 3}$. On the contrary, let us assume that ${\beta =2}$. Now from \eqref{xss5}, we get \begin{eqnarray}
		\nonumber \alpha (q_{ik}+1)&=&\beta (q'+1)\\\nonumber \text{i.e., }\quad \alpha (q_{ik}+1) &=& 2(q'+1)\\\nonumber &<& 2(q_{ik}+1),
	\end{eqnarray}
	which is a contradiction as ${\alpha \geq 2}$. Therefore, we must have ${\beta \geq 3}$, i.e., ${g''(z_0)=0}$. This proves (II).
\end{proof}
\begin{lem}\label{lem6}
	Let ${f}$ and ${g}$ be two non-constant meromorphic functions. For a polynomial $
	{P(z)}$, let $P(f)=P(g)$ and  ${t\geq  3}$. Therefore, there are at least three ${B_i(H_2)}$'s are non empty; say ${B_1(H_2),B_2(H_2),B_3(H_2)}$ are non empty.  Let ${q_i= \max  \{q:q\in A_i(H_2)\}}$. Let ${d_i(\in B_i(H_2))}$ corresponds to ${q_i}$ ${(\in A_i(H_2))}$, for ${i=1,2,3}$. Define \begin{eqnarray}\label{Psi}
		\Psi&=&\dfrac{f'g'(f-g)^2}{(f-d_1)(f-d_2)(f-d_3)(g-d_1)(g-d_2)(g-d_3)}.
	\end{eqnarray} 
	Then \begin{itemize}
		\item [(I)]   ${\Psi}$ is an entire function.
		\item[(II)] Let ${q_j\geq  2}$ for some $j\in \{1,2,3\}$.   Then \begin{eqnarray}
            \nonumber \overline{N}(r,d_j;f|g'=0,g\neq d_j)=S(r,f)\,\, and \,\,\overline{N}(r,d_j;g|f'=0,f\neq d_j)=S(r,f).
        \end{eqnarray}
        

	\end{itemize}
\end{lem}
\begin{proof}
	{\em \underline{Proof of (I).}}  Let ${f(z_0)=d_1}$ (or ${d_2}$ or ${d_3}$). Now \begin{itemize}
		\item [(i)] If ${g'(z_0)\neq 0}$, then form {\em Lemma \ref{lem4}}, we see that ${g(z_0)=d_1}$(or ${d_2}$ or ${d_3}$) and hence ${z_0}$ is not a pole of ${\Psi}$.
		\item [(ii)] If ${g'(z_0)}=0$ then  \begin{itemize}
			\item [(a)] if ${g(z_0)=d_1}$(or ${d_2}$ or ${d_3}$) then ${g(z_0)}$ is not a pole of ${\Psi}$; 
			\item [(b)] if ${g(z_0)\neq d_1}$(or ${d_2}$ or ${d_3}$), then ${g(z_0)\neq d_1,d_2,d_3}$ (see the paragraph just after {\em Table-1}). Therefore, ${z_0}$ is not a pole of ${\Psi}$.
		\end{itemize}
	\end{itemize}
	Now, if we consider $g(z_0)=d_1$ ($d_2$ or $d_3$), then by the similar arguments, we see that $z_0$ is not a pole of $\Psi$.
	Again, as ${P(f)=P(g)}$, so ${f,g}$ share ${ \infty }$ CM and hence poles of ${f,g}$ are not the poles of ${\Psi}$. Thus we find that ${\Psi}$ has no pole. Therefore, ${\Psi}$ is an entire function.
				\vspace{0.1in}\\{\em \underline{Proof of (II).}} Without loss of generality, we assume ${q_j=q_1\geq 2}$. Let ${f(z_0)=d_1}$ of multiplicity ${\alpha (\geq 1) }$; ${g'(z_0)=0}$; ${g(z_0)\neq d_1}$. Now, we have two possibilities.\par {\underline{\textbf{Case-1:}}} Let ${g(z_0)\not \in S}$. Then ${P'(g(z_0))\neq 0}$. Now, as ${q_1\geq  2}$, we have \begin{eqnarray}
					\nonumber\left. \dfrac{d^2}{dz^2}P(f)\right|_{z=z_0}&=&\left. \dfrac{d^2}{dz^2}P(g)\right|_{z=z_0}\\\nonumber \Longrightarrow f''(z_0)P'(d_1)+f'^2(z_0)P''(d_1)&=&g''(z_0)P'(g(z_0))+g'^2(z_0)P''(g(z_0))\\\nonumber \Longrightarrow 0&=&g''(z_0)P'(g(z_0))\\\nonumber \Longrightarrow 0&=&g''(z_0),
				\end{eqnarray}
				i.e., ${z_0}$ is a zero of ${g'}$ of multiplicity at least two.
				
				\par {\underline{\textbf{Case-2:}}} Let ${g(z_0) \in S}$. Let ${g(z_0)=d}(\in S)$ of multiplicity ${\beta (\geq 2) }$, for some ${d\neq d_1}$. Let ${q_1,q_d}$ (in {\em Table 2}) corresponds to ${d_1,d}$  (in {\em Table 1}) respectively.  Then as ${d_1\in C_1}$, we must have  ${d\in C_1\setminus \{d_1\}}$. Again, as ${P(d_1)=P(d)}$, we see that ${(z-d_1)^{q_1+1}(z-d)^{q_d+1}}$ is a factor of ${P(z)-P(d_1)}$. Therefore, we must have \begin{eqnarray}
					\nonumber P(z)-P(d_1)&=&(z-d_1)^{q_1+1}(z-d)^{q_d+1}Q(z),
				\end{eqnarray}
				where ${Q(z)}$ is a polynomial with ${Q(d),Q(d_1)\neq 0}$. Now from ${P(f)=P(g)}$, we have \begin{eqnarray}\label{xs17}
					P(f)-P(d_1)&=&P(g)-P(d_1)\\\nonumber \text{i.e.,}\quad (f-d_1)^{q_1+1}(f-d)^{q_d+1}Q(f) &=& (g-d_1)^{q_1+1}(g-d)^{q_d+1}Q(g).
				\end{eqnarray}
            Comparing the powers of $(z-z_0)$ on both sides of  \eqref{xs17}, we have 	
                 \begin{eqnarray}\label{xs18}
					\alpha (q_1+1)&=&\beta (q_d+1).
				\end{eqnarray}
				Here ${q_1>q_d}$, ${\alpha \geq 1}$ and ${\beta \geq 2}$. We claim that ${\alpha \geq 2}$. If possible, let ${\alpha =1}$. Then from \eqref{xs18}, we get\begin{eqnarray}
					\nonumber q_1+1&=&\beta (q_d+1)\\\nonumber &\geq & 2(q_d+1),
				\end{eqnarray}
				which is a contradiction as both ${d,d_1\in B_1(H_2)}$ and ${q_1=\max \{q:q\in A_1(H_2)\}}$. Hence we must have ${\alpha \geq 2}$, i.e., ${f'(z_0)=0}$. Thus we finally have ${f(z_0)=d_1 }$,  ${f'(z_0)=0}$, ${g'(z_0)=0,}$ ${g(z_0)\neq d_1}$ and ${q_d<q_1}$. Now by using {\em Lemma \ref{e4}-(II)}, we see that ${g''(z_0)=0}$.
				\par\vspace{0.1in} Thus, in any case, we see that ${z_0}$ is a zero of ${g'}$ of multiplicity at least ${2}$. Hence ${\Psi(z_0)=0}$. So, \begin{eqnarray}
					\nonumber \overline{N}(r,d_1;f|g'=0,g\neq d_1)&\leq  &  \overline{N}(r,0,\Psi)\\\nonumber&\leq  &T(r,\Psi)\\\nonumber&=&S(r,f).
				\end{eqnarray}
				Similarly, \begin{eqnarray}
					\nonumber \overline{N}(r,d_1;g|f'=0,f\neq d_1)&= & S(r,g).
				\end{eqnarray}
				\par This completes the proof of the lemma.
			\end{proof}


	   \section*{\Large {Proof of Theorem \ref{t1} and Theorem \ref{t2}}}
	    At first we prove {\em Theorem \ref{t1}}.
	   \par Let us consider the following two cases. \par \textbf{Case-1:} Let ${\Phi=0}$. Then ${f=g}$. \vspace{0.1in} \par \textbf{Case-2:} Let ${\Phi\neq 0}$. Using {\em Lemma \ref{lem4}}, we get \begin{eqnarray}
	    	\nonumber  \sum\limits_{d\in \left({\cup_{i}^{} B_i(H_2)}\right)} \overline{N}(r,d;f)+ \overline{N}(r, \infty;f)&=& \sum\limits_{d\in \left({\cup_{i}^{} B_i(H_2)}\right)} \overline{N}(r,d;f|g'\neq 0)\\\nonumber&&+\sum\limits_{d\in \left({\cup_{i}^{} B_i(H_2)}\right)} \overline{N}(r,d;f|g'= 0)+ \overline{N}(r, \infty;f)\\\nonumber&\leq&   \overline{N}(r, 0;\Phi)+ \overline{N}(r,0;g')\\\nonumber&\leq  &\left({T(r,f)+T(r,g)}\right)+2T(r,g)+O(1)\\\nonumber&\leq  &4T(r,f)+O(1).
	    \end{eqnarray}
	    Now, using second fundamental theorem, we have \begin{eqnarray}
	    	\nonumber (t'-1)T(r,f)&\leq  & \sum\limits_{d\in \left({\cup_{i}^{} B_i(H_2)}\right)} \overline{N}(r,d;f)+ \overline{N}(r, \infty;f)- \log  r+O(1)\\\nonumber&\leq  &4T(r,f)- \log  r+O(1),
	    \end{eqnarray}
	    which is a contradiction, as ${t'\geq  5}$.\par This completes the proof of the theorem.  
If we apply the complex version of the second fundamental theorem, we get {\em Theorem \ref{t2}}.
\begin{rem}
   Let ${f,g}$ be entire functions and ${\alpha \not\in \{z:P(z)-P(d)=0\}}$, for all ${ d\in  \cup_{i}^{} B_i(H_2)}$. Let ${f(z_0)=d}$. Since ${P(f)=P(g)}$, we must have ${g(z_0)\in \{z:P(z)-P(d)=0\}}$. Therefore,
       \begin{eqnarray}\label{xvc1}
       \sum\limits_{d\in \left({\cup_{i}^{} B_i(H_2)}\right)} \overline{N}(r,d;f|g'= 0)&\leq &\overline{N}(r,0;\dfrac{g'}{g-\alpha } )\\\nonumber&\leq  & T(r,\dfrac{g'}{g-\alpha })+O(1)\\\nonumber&\leq  & N(r,\dfrac{g'}{g-\alpha } )+S(r,g)\\\nonumber&\leq  & \overline{N}(r,\alpha ;g)+S(r,g)\\\nonumber&\leq  & T(r,g)+S(r,g).
   \end{eqnarray}
Incorporating \eqref{xvc1} in the proof of {\em Theorem \ref{t1}} and {\em Theorem \ref{t2}}, we arrive at a contradiction for ${t'\geq 4}$ and ${t'\geq 5}$ respectively.
\end{rem}
	   \section*{\Large {Proof of Theorem \ref{t3}}}
	We consider the following two cases.
	\par \textbf{\underline{\textbf{Case-1:}}} Let ${\Phi=0}$. Then ${f=g}$.
	\par 	\textbf{\underline{\textbf{Case-2:}}} Let ${\Phi\neq 0}$. Then ${f\neq g}$.
	Since ${t\geq  3}$, let us assume that ${B_1(H_2),B_2(H_2),B_3(H_2)}$ are non empty. Let ${d_i\in B_i(H_2)}$ such that ${q_i= \max  \{q:q\in A_i(H_2)\}}$ for ${i=1,2,3}$, where ${q_i}$ in {\em Table 2} is the corresponding element of ${d_i}$ in {\em Table 1}.
	 Let ${ \overline{N}_0(r,0;g')}$ be the reduced counting function of those zeros of ${g'}$, which are not the zeros of ${g-d_i}$ for ${i=1,2,3}$. ${ \overline{N}_0(r,0;f')}$ is similarly defined. Then \begin{eqnarray}\label{ee1}
		 \sum\limits_{j=1}^{3} \overline{N}(r,d_i;f|g'= 0,g\neq d_i)&\leq  &  \overline{N}_0(r,0;g')
	\end{eqnarray}
and 
\begin{eqnarray}\label{ee11}
	\sum\limits_{j=1}^{3} \overline{N}(r,d_i;g|f'= 0,f\neq d_i)&\leq  &  \overline{N}_0(r,0;f').
\end{eqnarray}
	 Now, using {\em Lemma \ref{lem4}} and \eqref{ee1} in the second fundamental theorem, we have \begin{eqnarray}\label{ee2}
	 	 2T(r,f)&\leq  & \sum\limits_{j=1}^{3} \overline{N}(r,d_i;f)+ \overline{N}(r, \infty;f)-N(r,0;f')- \log  r+O(1)\\\nonumber&\leq  &   \sum\limits_{j=1}^{3}\overline{N}(r,d_i;f|g'\neq 0)+  \sum\limits_{j=1}^{3}\overline{N}(r,d_i;f|g'=0,g=d_i) \\\nonumber&&+\sum\limits_{j=1}^{3} \overline{N}(r,d_i;f|g'=0;g\neq d_i)+ \overline{N}(r, \infty;f)-N_0(r,0;f')\\\nonumber&&- \log  r+O(1)\\\nonumber&\leq  & \overline{N}(r,0;\Phi)+ \overline{N}_0(r,0;g')-N_0(r,0;f')- \log  r+O(1).
	 \end{eqnarray}
	 Similarly, \begin{eqnarray}\label{ee3}
	 	2T(r,g)&\leq  & \overline{N}(r,0;\Phi)+ \overline{N}_0(r,0;f')-N_0(r,0;g')- \log  r+O(1).
	 \end{eqnarray}
	  Now, combining \eqref{ee2} and \eqref{ee3}, we get \begin{eqnarray}
	  	\nonumber 2\left[{T(r,f)+T(r,g)}\right]&\leq  &2 \overline{N}(r,0;\Phi)-2 \log  r+O(1)\\\nonumber&\leq  &2\left[{T(r,f)+T(r,g)}\right]-2 \log  r+O(1),
	  \end{eqnarray}
	  which is a contradiction.
	  \par This completes the proof of the theorem.
	   \section*{\Large {Proof of Theorem \ref{t4}}}
 If ${t\geq  4}$, the proof is almost similar to the proof of {\em Theorem \ref{t3}}. Therefore, we prove the theorem only when ${t=3}$.\par 
	Since ${t=  3}$, let us assume ${B_1(H_2),B_2(H_2),B_3(H_2)}$ be non empty. Let ${d_i\in B_i(H_2)}$ such that ${q_i=\ \max  \{q:q\in A_i(H_2)\}}$ for ${i=1,2,3}$, where ${q_i}$ in {\em Table 2} corresponds to ${d_i}$ in {\em Table 1}. Now we consider the following two cases. 
	\vspace{0.1in}\\{\underline{\textbf{Case-1:}}} Let ${\Phi=0}$. Then ${f=g}$.
	\vspace{0.1in}\\{\underline{\textbf{Case-2:}}} Let ${\Phi\neq 0}$. Then ${f\neq g}$.
 From {\em Lemma \ref{lem6}}, we see that the function ${\Psi}$ for the points ${d_1,d_2,d_3}$, i.e., \begin{eqnarray}
 	\nonumber \Psi&=& \dfrac{f'g'(f-g)^2}{(f-d_1)(f-d_2)(f-d_3)(g-d_1)(g-d_2)(g-d_3)}
 \end{eqnarray}
 is an entire function, with \begin{eqnarray}\label{e5}
 	T(r,\Psi)&=&S(r,f).
 \end{eqnarray}
Since ${t'\geq  4}$, there is a ${d(\neq d_1,d_2,d_3)\in B_i(H_2)}$ for some ${i=1,2,3}$. Without loss of generality, let us assume ${d\in B_1(H_2)}$. We may choose ${d\in B_1(H_2)}$ in such a way that ${q_1>q_d}$ but $q_d\geq q_\beta$
 for any ${q_ \beta \in A_1(H_2)}$, unless otherwise ${q_ \beta =q_1}$.
 Now \begin{eqnarray}
 	\nonumber  &&\overline{N}(r,d;f|g'=0)\\\nonumber=&& \overline{N}(r,d;f|g'=0;g=d)+ \overline{N}(r,d;f|g'=0;g\neq d)\\\nonumber=&&N(1)+\overline{N}(r,d;f|g'=0;g\neq d;f'\neq 0)+\overline{N}(r,d;f|g'=0;g\neq d;f'= 0)\\\nonumber=&&N(1)+N(2)+\overline{N}(r,d;f|g'=0;g\neq d;f'= 0;g\not\in B(H_2))+\overline{N}(r,d;f|g'=0;g\neq d;f'= 0;g\in B(H_2))\\\nonumber=&&N(1)+N(2)+N(3)+N(4), 
 \end{eqnarray}
 where \begin{itemize}
 	\item ${N(1)= \overline{N}(r,d;f|g'=0,g=d)}$; \item ${N(2)= \overline{N}(r,d;f|g'=0;g\neq d;f'\neq 0)}$; \item ${N(3)=\overline{N}(r,d;f|g'=0;g\neq d;f'= 0;g\not \in B(H_2))}$; \item ${N(4)=\overline{N}(r,d;f|g'=0;g\neq d;f'= 0;g \in B(H_2))}$.
 \end{itemize}
Again, \begin{eqnarray}\label{ee8}
	 \nonumber&&\overline{N}(r,d_i;f|g'=0;g\neq d_i)\\\nonumber=&& \overline{N}(r,d_i;f|g'=0;g\neq d_i;f'\neq 0)+\overline{N}(r,d_i;f|g'=0;g\neq d_i;f'=0)\\\nonumber=&& N(5)+\overline{N}(r,d_i;f|g'=0;g\neq d_i;f'=0;g\in B(H_2))+\overline{N}(r,d_i;f|g'=0;g\neq d_i;f'=0;g\not\in B(H_2))\\\nonumber=&&N(5)+N(6)+N(7),
\end{eqnarray}
where
\begin{itemize}
	\item $N(5)={\overline{N}(r,d_i;f|g'=0;g\neq d_i;f'\neq 0)}$ \item $N(6)={\overline{N}(r,d_i;f|g'=0;g\neq d_i;f'=0;g\in B(H_2))}$ \item ${N(7)=\overline{N}(r,d_i;f|g'=0;g\neq d_i;f'=0;g\not\in B(H_2))}$
\end{itemize}
Observe that the points, whose counting function is \begin{itemize}
	\item ${N(1)}$, are zeros of ${\Phi}$.
	\item ${N(2),}$ are zeros of ${g'}$, which are not the zeros of ${g-d_ \alpha }$ for all ${d_ \alpha \in B(H_2)}$ (by using {\em Lemma \ref{e2}}). \item ${N(3)}$, are zeros of ${g'}$, which are not the zeros of ${g-d_ \alpha }$ for all ${d_ \alpha \in B(H_2)}$.
	\item ${N(4)}$,  are  zeros of ${\Psi(z)}$. \par {\underline{{\em Explanation:}}} Let ${f(z_0)=d}$; ${g'(z_0)=0}$; ${g(z_0)\neq d}$; ${f'(z_0)=0}$; ${g(z_0)\in B(H_2)}$. Then as ${d\in C_1}$, we must have ${g(z_0)\in C_1}$.\par If ${g(z_0)=d_1}$, then as ${q_d<q_1}$, from {\em Lemma \ref{e4}-(I)}, we have ${f''(z_0)=0}$, i.e., ${z_0}$ is a zero of  ${f'}$ of multiplicity at least two. Then ${\Psi(z_0)=0}$.  \par If ${g(z_0)\neq d,d_1}$, then ${g(z_0)\neq d_1,d_2,d_3}$ and hence ${\Psi(z_0)=0}$. 
	\item ${N(5)}$, are zeros of ${g'}$, which are not the zeros of ${g-d_ \alpha }$ for all ${d_ \alpha \in B(H_2)}$(by using {\em Lemma \ref{e2}}).
	\item ${N(6)}$, are zeros of ${\Psi(z)}$. 
	\par {\underline{{\em Explanation:}}}  Let ${f(z_0)=d_i}$; with ${g'(z_0)=0}$; ${g(z_0)\in B(H_2)}$; ${f'(z_0)=0}$; ${g(z_0)\neq d_i}$; ${i=1,2,3}$. Let ${g(z_0)=d_i'}$ for some ${d_i'\in C_i}$; with ${d_i'\neq d_i}$; ${i=1,2,3}$. Let ${q_i'}$ be the corresponding element of ${d_i'}$. Clearly, ${q_i>q_i'}$. Therefore, proceeding similarly like {\em Lemma \ref{e4}(Case-1)}, we get ${g''(z_0)=0}$; i.e., ${z_0}$ is a zero of ${g'}$ of multiplicity at least two. Thus ${\Psi (z_0)=0}$.
	\item ${N(7)}$, are zeros of ${g'}$, which are not the zeros of ${g-d_ \alpha }$ for all ${d_ \alpha \in B(H_2)}$. 
\end{itemize}
Now, using {\em Lemma \ref{lem4}} \begin{eqnarray}
	\nonumber   \sum\limits_{i=1}^{3} \overline{N}(r,d_i;f)+\overline{N}(r,d;f)+ \overline{N}(r, \infty;f)&\leq  & \sum\limits_{i=1}^{3} \overline{N}(r,d_i;f|g'\neq 0)+\sum\limits_{i=1}^{3} \overline{N}(r,d_i;f|g'=0)\\\nonumber&&+ \overline{N}(r,d;f|g'=0)+ \overline{N}(r,d;f|g'\neq 0)+ \overline{N}(r, \infty;f)\\\nonumber&=& \sum\limits_{i=1}^{3} \overline{N}(r,d_i;f|g'\neq 0)+ \sum\limits_{i=1}^{3} \overline{N}(r,d_i;f|g'=0;g=d_i)\\\nonumber&&+\sum\limits_{i=1}^{3} \overline{N}(r,d_i;f|g'=0;g\neq d_i)+ \overline{N}(r,d;f|g'=0)\\\nonumber&&+ \overline{N}(r,d;f|g'\neq 0)+ \overline{N}(r, \infty;f)\\\nonumber&\leq  & \sum\limits_{i=1}^{3} \overline{N}(r,d_i;f|g'\neq 0)+ \sum\limits_{i=1}^{3} \overline{N}(r,d_i;f|g'=0;g=d_i)\\\nonumber&&+\left[{N(5)+N(6)+N(7)}\right]+ \left[{N(1)+N(2)+N(3)+N(4)}\right]\\\nonumber&&+  \overline{N}(r,d;f|g'\neq 0)+ \overline{N}(r, \infty;f)\\\nonumber&\leq  & \overline{N}(r,0;\Phi)+ \overline{N}(r,0;g'|g\not\in B(H_2))+ \overline{N}(r,0;\Psi)\\\nonumber&\leq  & \left[{T(r,f)+T(r,g)}\right]+ \overline{N}_0(r,0;g')+S(r,f).
\end{eqnarray}
Now by the second fundamental theorem, we have \begin{eqnarray}\label{z1xz}
	 3T(r,f)&\leq  & \sum\limits_{i=1}^{3} \overline{N}(r,d_i;f)+ \overline{N}(r,d;f)+ \overline{N}(r, \infty;f)-N_0(r,0;f')+S(r,f)\\\nonumber&\leq  &\left[{T(r,f)+T(r,g)}\right]+ \overline{N}_0(r,0;g')-N_0(r,0;f')+S(r,f).
\end{eqnarray}
Similarly, \begin{eqnarray}\label{z2xz}
	  3T(r,g)&\leq  &\left[{T(r,f)+T(r,g)}\right]+ \overline{N}_0(r,0;f')-N_0(r,0;g')+S(r,g).
\end{eqnarray}
Combining \eqref{z1xz} and \eqref{z2xz}, we have \begin{eqnarray}
	\nonumber 3\left[{T(r,f)+T(r,g)}\right]&\leq  &2\left[{T(r,f)+T(r,g)}\right]+S(r,f)+S(r,g),
\end{eqnarray}
which is a contradiction.
\par This completes the proof of the theorem.
	   \section*{\Large {Proof of Theorem \ref{t7}}}
If ${\min (q_1,q_2,q_3)\geq  2}$, then from {\em Lemma \ref{lem1}, Lemma \ref{lem4}} and {\em  Lemma \ref{lem6}}, we see that ${f,g}$ share ${d_1,d_2,d_3}$ \textbf{``CM"}. Again, ${P(f)=P(g)}$ implies that ${f,g}$ share ${ \infty }$ CM. Hence by the four value theorem \cite{nevanlinna}, we get ${g=T(f)}$ for some Mobius Transformation ${T}$.\par However, if ${\max  (q_1,q_2,q_3)\geq  2}$, then under the hypothesis assumed in the theorem, we shall show that same conclusion holds; i.e., ${g=T(f)}$.
\par\vspace{0.1in} Since ${ \max(q_1,q_2,q_3)\geq  2}$, without loss of generality, let us assume ${q_1\geq  2}$. Let \begin{eqnarray}
	\nonumber \Psi&=&\dfrac{f'g'(f-g)^2}{(f-d_1)(f-d_2)(f-d_3)(g-d_1)(g-d_2)(g-d_3)};\\\nonumber \gamma&=  &\dfrac{f'}{f-d_1}-\dfrac{g'}{g-d_1}.
\end{eqnarray}
Using {\em Lemma \ref{lem6}}, we have 
\begin{eqnarray}\label{x9}
	\overline{N}(r,d_1;f|g'=0,g\neq d_1)&=& S(r,f)
\end{eqnarray}
and \begin{eqnarray}\label{x11}
	\overline{N}(r,d_1;g|f'=0,f\neq d_1)&=& S(r,g),
\end{eqnarray}
i.e., ${f,g}$ share ${d_1}$ \textbf{``CM"}.
Therefore, using {\em Lemma \ref{lem4}}, we see that \begin{eqnarray}\label{x25}
	\overline{N}(r,d_1;f)&=& \overline{N}(r,d_1;g)+S(r,f).
\end{eqnarray}
Again, by using {\em Lemma \ref{lem4}}, {\em Lemma \ref{lem1}}, \eqref{x9} and \eqref{x11}, we have \begin{eqnarray}\label{x10}
	T(r, \gamma  )&=&N(r, \gamma  )+m(r, \gamma  )\\\nonumber&=& \overline{N}(r,d_1;f|g'=0,g\neq d_1)+\overline{N}(r,d_1;g|f'=0,f\neq d_1)\\\nonumber&&+S(r,f)+S(r,g)\\\nonumber&=&S(r,f).
\end{eqnarray}
Again, from {\em Lemma \ref{lem6}}, we see that \begin{eqnarray}\label{x12}
	T(r,\Psi)&=&S(r,f).
\end{eqnarray}
Let ${A_i,B_i\in \mathbb{C}}$ be constants with ${A_1,B_1\neq 0}$, and \begin{eqnarray}\label{x1}
	f(z)&=&d_2+(z-z_0)A_1+(z-z_0)^2A_2+ \dots \\g(z)&=&d_2+(z-z_0)B_1+(z-z_0)^2B_2+ \dots.
\end{eqnarray}
Then \begin{eqnarray}\label{x3}
	\Psi(z_0)&=&\dfrac{(A_1-B_1)^2}{(d_1-d_2)^2(d_2-d_3)^2},
\end{eqnarray}
and \begin{eqnarray}\label{x4}
	\gamma(z_0) &=&\dfrac{A_1-B_1}{d_2-d_1}. 
\end{eqnarray}
Now from \eqref{x3} and \eqref{x4}, we see that \begin{eqnarray}\label{x5}
	\left({\gamma (z_0)}\right)^2&=&(d_2-d_3)^2\Psi(z_0).
\end{eqnarray}
Similarly, if we assume \begin{eqnarray}
	\nonumber f(z)&=&d_3+(z-z_0)A_1+(z-z_0)^2A_2+ \dots \\\nonumber  g(z)&=&d_3+(z-z_0)B_1+(z-z_0)^2B_2+ \dots,
\end{eqnarray}
still we get \eqref{x5}.
Now, we have the following two cases to consider.\par {\underline{\textbf{Case-1:}}} Let \begin{eqnarray}\label{x6}
	 \gamma  ^2&\neq& (d_2-d_3)^2\Psi.
\end{eqnarray}
Now, using {\em Lemma \ref{lem1}}, we have from \eqref{x5} \begin{eqnarray}\label{x13}
	&&\overline{N}(r,d_2;f|f'\neq0, g=d_2)+\overline{N}(r,d_3;f|f'\neq0, g=d_3)\\\nonumber&=&\overline{N}(r,d_2;f|f'\neq0, g=d_2,g'\neq 0)+\overline{N}(r,d_3;f|f'\neq0, g=d_3,g'\neq 0)\\\nonumber&\leq  & \overline{N}(r,0; \gamma  ^2-(d_2-d_3)^2\Psi )\\\nonumber&\leq  &2T(r, \gamma  )+T(r,\Psi )\\\nonumber&=&S(r,f).
\end{eqnarray}
Again, by using {\em Lemma \ref{lem1}} \begin{eqnarray}\label{x14}
	\overline{N}(r, d_2;f|g=d_2,f'=0)&=  &  \overline{N}(r,d_2;f|g=d_2,f'=0,g'=0)\\\nonumber&\leq  & \overline{N}(r,0;\Psi )\\\nonumber&\leq  &T(r,\Psi)+O(1)\\\nonumber&=  &S(r,f).
\end{eqnarray}
Similarly,\begin{eqnarray}\label{x15}
	\overline{N}(r, d_3;f|g=d_3,f'=0)&=  &  \overline{N}(r,d_3;f|g=d_3,f'=0,g'=0)\\\nonumber&= & S(r,f).
\end{eqnarray}
Using \eqref{x13} and \eqref{x14}, \begin{eqnarray}\label{x16}
	\overline{N}(r,d_2,f|g=d_2)&=&S(r,f).
\end{eqnarray}
Similarly, \begin{eqnarray}\label{x17}
	\overline{N}(r,d_3,f|g=d_3)&=&S(r,f).
\end{eqnarray}
Using {\em Lemma \ref{lem4}}, we have \begin{eqnarray}\label{x18}
	&&\sum\limits_{j=2}^{3}\overline{N}(r,d_j;f|g\neq d_j)\\\nonumber&=& \sum\limits_{j=2}^{3}\overline{N}(r,d_j;f|g\neq d_j,g'=0)+\sum\limits_{j=2}^{3}\overline{N}(r,d_j;f|g\neq d_j,g'\neq 0)\\\nonumber&=&\sum\limits_{j=2}^{3}\overline{N}(r,d_j;f|g\neq d_j,g'=0).
\end{eqnarray}
Similarly, \begin{eqnarray}\label{x45}
	\sum\limits_{j=2}^{3} \overline{N}(r,d_j;g|f\neq d_j)&=&\sum\limits_{j=2}^{3}\overline{N}(r,d_j;g|f\neq d_j,f'=0).
\end{eqnarray}
Let ${ \overline{N}_0(r,0;g')}$ be the reduced counting function of those zeros of ${g'}$ which are not the zeros of ${g-d_i}$, ${i=1,2,3}$.
Therefore, from \eqref{x16}, \eqref{x17} and \eqref{x18}, we get \begin{eqnarray}\label{x19}
	\sum\limits_{i=2}^{3}\overline{N}(r,d_i;f)&\leq  &  \overline{N}_0(r,0;g')+S(r,f).
\end{eqnarray}
Similarly,  \begin{eqnarray}\label{x20}
	\sum\limits_{i=2}^{3}\overline{N}(r,d_i;g)&\leq  &  \overline{N}_0(r,0;f')+S(r,f).
\end{eqnarray}
Now, by the second fundamental theorem, we have \begin{eqnarray}\label{x21}
	2T(r,f)&\leq  & \sum\limits_{i=1}^{3} \overline{N}(r,d_i;f)+ \overline{N}(r, \infty;f)-N_0(r,0;f')+S(r,f)
\end{eqnarray}
and 
\begin{eqnarray}\label{x22}
	2T(r,g)&\leq  & \sum\limits_{i=1}^{3} \overline{N}(r,d_i;g)+ \overline{N}(r, \infty;g)-N_0(r,0;g')+S(r,g).
\end{eqnarray}
Adding \eqref{x21}, \eqref{x22} and then using \eqref{x19}, \eqref{x20} and \eqref{x25}, we get \begin{eqnarray}\label{x23}
	2\left[{T(r,f)+T(r,g)}\right]&\leq  & \overline{N}(r,d_1;f)+ \overline{N}(r,d_1;g)+ \overline{N}(r, \infty;f)\\\nonumber&&+\overline{N}(r, \infty;g)+S(r,f)\\\nonumber i.e., \qquad 4T(r,f)&\leq  &2 \overline{N}(r,d_1;f)+ 2\overline{N}(r, \infty;f)+S(r,f)\\\label{x26} i.e., \qquad 2T(r,f)&\leq  & \overline{N}(r,d_1;f)+ \overline{N}(r, \infty;f)+S(r,f).
\end{eqnarray}
Therefore, we must have \begin{eqnarray}\label{x27}
	\overline{N}(r,d_1;f)=T(r,f)+S(r,f) \qquad \& \qquad  \overline{N}(r, \infty;f)=T(r,f)+S(r,f)	
\end{eqnarray}
and hence from \eqref{x25}, we also have \begin{eqnarray}\label{x34}
	\overline{N}(r,d_1;g)=T(r,g)+S(r,f) \qquad \& \qquad  \overline{N}(r, \infty;g)=T(r,g)+S(r,g).
\end{eqnarray}
For any ${b_i\in \mathbb{S}}=\{z:P(z)=0\}=\{b_1,b_2, \dots,b_n\}$, let ${f(z_0)=b_i}$. Then as ${P(f)=P(g)}$, we must have $g(z_0)=b_j$ for some ${b_j\in \mathbb{S}}$.
Now we define \begin{eqnarray}\label{x31}
	\delta &=&\dfrac{f-d_1}{g-d_1}.
\end{eqnarray}
Clearly, ${\delta\not=0 }$. Since ${f,g}$ share ${ \infty }$ CM and ${d_1}$ \textbf{``CM"}, we have 
\begin{eqnarray}\label{x32}
	N(r,\delta)&=&S(r,f).
\end{eqnarray}
Again, from  \eqref{x27} and \eqref{x31}, we see that \begin{eqnarray}\label{x33}
	m(r,\delta)&\leq  & m(r,f)+m(r,d_1;g)+O(1).\\\nonumber&=&S(r,f).
\end{eqnarray}
So, \begin{eqnarray}\label{x34}
	T(r,\delta)&=&S(r,f).
\end{eqnarray}
Now \begin{eqnarray}\label{x35}
	\overline{N}(r,b_i;f)&\leq  & \sum\limits_{j=1}^{n} \overline{N}\left(r,\dfrac{b_i-d_1}{b_j-d_1};\delta\right)\\\nonumber&\leq  &nT(r,\delta)\\\nonumber&=&S(r,f).
\end{eqnarray}
Again, by the second fundamental theorem, we get
\begin{eqnarray}
	\nonumber (n-2)T(r,f)&\leq  & \sum\limits_{i=1}^{n} \overline{N}(r,b_i;f)+S(r,f)\\\nonumber&=&S(r,f),
\end{eqnarray}
which is a contradiction as ${n\geq  5}$.


\par {\underline{\textbf{Case-2:}}} Let \begin{eqnarray}\label{x46}
	 \gamma  ^2&=& (d_2-d_3)^2\Psi.
\end{eqnarray}
\begin{proposition}\label{propo1}
   $ \gamma  ^2-(d_2-d_3)^2 \Psi = \zeta _1\zeta _2$, where $\zeta _1=\dfrac{f'(f-d_3)}{(f-d_1)(f-d_2)}-\dfrac{g'(g-d_3)}{(g-d_1)(g-d_2)}$ and $\zeta _2=\dfrac{f'(f-d_2)}{(f-d_1)(f-d_3)}-\dfrac{g'(g-d_2)}{(g-d_1)(g-d_3)}$.

\end{proposition}
    \par\noindent \textit{Proof.}  Here \begin{eqnarray}
        \nonumber  \zeta _1\zeta _2&=&\left[ \dfrac{f'(f-d_3)(g-d_1)(g-d_2)-g'(g-d_3)(f-d_1)(f-d_2)}{(f-d_1)(f-d_2)(g-d_1)(g-d_2)} \right]\\\nonumber &&\times \left[ \dfrac{f'(f-d_2)(g-d_1)(g-d_3)-g'(g-d_2)(f-d_1)(f-d_3)}{(f-d_1)(f-d_3)(g-d_1)(g-d_3)} \right].
    \end{eqnarray}
    So, \begin{eqnarray}\label{nbh1}
        &&(f-d_1)(g-d_1) \prod\limits _{j=1}^{3}(f-d_j)(g-d_j)\zeta _1\zeta _2\\\nonumber =&&\left[f'(f-d_3)(g-d_1)(g-d_2)-g'(g-d_3)(f-d_1)(f-d_2)\right]\\\nonumber&&\times\left[f'(f-d_2)(g-d_1)(g-d_3)-g'(g-d_2)(f-d_1)(f-d_3)\right]\\\nonumber= &&f'^2\left[ (g-d_1)^2\prod\limits_{j=2}^{3}(f-d_j)(g-d_j) \right]+g'^2\left[ (f-d_1)^2\prod\limits_{j=2}^{3}(f-d_j)(g-d_j) \right]\\\nonumber &&-f'g'\left[(f-d_1)(g-d_1)\left\{(f-d_3)^2(g-d_2)^2+(f-d_2)^2(g-d_3)^2 \right\} \right]\\\nonumber =&&f'^2A_1+g'^2B_1-f'g'C_1,
    \end{eqnarray}
    where
    \begin{eqnarray}\label{nbh2}
        \left\{\begin{array}{l}
            A_1=(g-d_1)^2\prod\limits_{j=2}^{3}(f-d_j)(g-d_j),\\B_1=(f-d_1)^2\prod\limits_{j=2}^{3}(f-d_j)(g-d_j),\\C_1=(f-d_1)(g-d_1)\left\{(f-d_3)^2(g-d_2)^2+(f-d_2)^2(g-d_3)^2 \right\}.
        \end{array}\right.
    \end{eqnarray}
    Again, \begin{eqnarray}
        \nonumber  \gamma  ^2-(d_2-d_3)^2\Psi &=&\dfrac{\left( f'(g-d_1)-g'(f-d_1) \right)^2}{(f-d_1)^2(g-d_1)^2}-\dfrac{(d_2-d_3)^2f'g'(f-g)^2}{\prod\limits _{j=1}^{3}(f-d_j)(g-d_j)}.
    \end{eqnarray}
    So,
    \begin{eqnarray}\label{nbh4*}
        &&(f-d_1)(g-d_1) \prod\limits _{j=1}^{3}(f-d_j)(g-d_j)\left[ \gamma  ^2-(d_2-d_3)^2\Psi  \right]\\\nonumber= &&(f'(g-d_1)-g'(f-d_1))^2\prod\limits_{j=2}^{3}(f-d_j)(g-d_j)-(d_2-d_3)^2(f-d_1)(g-d_1)f'g'(f-g)^2\\\nonumber= &&\left[f'^2(g-d_1)^2+g'^2(f-d_1)^2-2f'g'(f-d_1)(g-d_1)\right]\prod\limits_{j=2}^{3}(f-d_j)(g-d_j)\\\nonumber&&-(d_2-d_3)^2f'g'(f-g)^2(f-d_1)(g-d_1)\\\nonumber= &&f'^2\left[ (g-d_1)^2\prod\limits_{j=2}^{3}(f-d_j)(g-d_j) \right]+g'^2\left[ (f-d_1)^2\prod\limits_{j=2}^{3}(f-d_j)(g-d_j) \right]\\\nonumber &&-f'g'\left[ 2(f-d_1)(g-d_1)\prod\limits_{j=2}^{3}(f-d_j)(g-d_j)-(d_2-d_3)^2(f-g)^2(f-d_1)(g-d_1) \right]\\\nonumber =&&A_2f'^2+B_2g'^2-C_2f'g',
    \end{eqnarray}
    where
    \begin{eqnarray}\label{nbh3}
        \left\{\begin{array}{l}
            A_2=(g-d_1)^2\prod\limits_{j=2}^{3}(f-d_j)(g-d_j),\\B_2=(f-d_1)^2\prod\limits_{j=2}^{3}(f-d_j)(g-d_j),\\C_2=2(f-d_1)(g-d_1)\prod\limits_{j=2}^{3}(f-d_j)(g-d_j)\\\quad\quad\quad-(d_2-d_3)^2(f-g)^2(f-d_1)(g-d_1).
        \end{array}\right.
    \end{eqnarray}
    Now from \eqref{nbh2} and \eqref{nbh3}, we see that ${A_1=A_2}$, ${B_1=B_2}$, and it is easy to verify that \begin{eqnarray}
        \nonumber  C_1\,\,=\,\,C_2&=&2f^2g^2-2(d_2+d_3)fg(f+g)+(d_2^2+d_3^2)(f^2+g^2)+8d_2d_3fg\\\nonumber&&-2d_2d_3(d_2+d_3)(f+g)+2d_2^2d_3^2.
    \end{eqnarray}
Therefore, from \eqref{nbh1} and \eqref{nbh4*}, we have \begin{eqnarray}
        \nonumber  (f-d_1)(g-d_1) \prod\limits _{j=1}^{3}(f-d_j)(g-d_j)\zeta _1\zeta _2&=&(f-d_1)(g-d_1) \prod\limits _{j=1}^{3}(f-d_j)(g-d_j)\left[ \gamma  ^2-(d_2-d_3)^2\Psi  \right]\\\nonumber \text{ i.e., }\quad\zeta _1\zeta _2&=&\gamma  ^2-(d_2-d_3)^2\Psi.
    \end{eqnarray}
    


This completes the proof of {\em Proposition \ref{propo1}}. 
\par Now, \eqref{x46} and {\em Proposition \ref{propo1}} together implies ${\zeta _1\zeta _2=0}$, i.e., either ${\zeta_1=0}$ or ${\zeta_2=0}$. To establish ${f=g}$, we shall consider the case when ${\zeta_1=0}$ only. The proof for the case when ${\zeta_2=0}$ is similar.
\par\vspace{0.2in} Let ${\zeta _1= 0}$. Then \begin{eqnarray}\label{z4}
	\dfrac{f'(f-d_3)}{(f-d_1)(f-d_2)}&=&\dfrac{g'(g-d_3)}{(g-d_1)(g-d_2)}.
\end{eqnarray}
Since $d_1,d_2$ lies in two different columns, so $d_1$ points of $f$ (or $g$) cannot be $d_2$ points of $g$ (or $f$), and $d_2$ points of $f$ (or $g$) cannot be $d_1$ points of $g$ (or $f$). Again, $P(f)=P(g)$ implies $f,g$ share $\infty$ CM. So, from {\em Lemma \ref{lem1}} and \eqref{z4}, we see that $f,g$ share $d_1,d_2,\infty$ CM.
\begin{proposition}\label{propo2}
	${f,g}$ share ${d_3}$ \textbf{``CM"}.
\end{proposition}
\par\noindent {\em Proof.} If ${q_3\geq 2}$, then from {\em Lemma \ref{lem6}}, we see that ${f,g}$ share ${d_3}$ \textbf{``CM"} and the {\em Proposition} holds trivially. Thus we shall assume ${q_3=1}$. From \eqref{z4}, we have \begin{eqnarray}\label{jj1}
	\dfrac{f'(f-d_3)}{g'(g-d_3)}&=&\dfrac{(f-d_1)(f-d_2)}{(g-d_1)(g-d_2)}.
\end{eqnarray}
Again, as ${q_3=1}$ and ${f'P'(f)=g'P'(g)}$, we have \begin{eqnarray}\label{jj2}
	\dfrac{f'(f-d_3)}{g'(g-d_3)}&=&\dfrac{(g-d_1)^{q_1}(g-d_2)^{q_2}Q(g)}{(f-d_1)^{q_1}(f-d_2)^{q_2}Q(f)}.
\end{eqnarray}
Thus, from \eqref{jj1} and \eqref{jj2}, we get
\begin{eqnarray}\label{jj3}
	(f-d_1)^{q_1+1}(f-d_2)^{q_2+1}Q(f)&=&(g-d_1)^{q_1+1}(g-d_2)^{q_2+1}Q(g).
\end{eqnarray}
Now differentiating \eqref{jj3} and then using ${f'P'(f)=g'P'(g)}$, we get
\begin{eqnarray}\label{jj4}
	&&\dfrac{Q'(f)(f-d_1)(f-d_2)-Q(f)\left[{d_2(1+q_1)+d_1(1+q_2)-(2+q_1+q_2)f}\right]}{(f-d_3)Q(f)}\\\nonumber=&&\dfrac{Q'(g)(g-d_1)(g-d_2)-Q(g)\left[{d_2(1+q_1)+d_1(1+q_2)-(2+q_1+q_2)g}\right]}{(g-d_3)Q(g)}.
\end{eqnarray}
\par \textbf{\underline{Case-2.1:}} Let ${P(z)}$ be CIP. Now if ${Q(z)}$ is non-constant, then we have ${t\geq  3}$ and ${t'\geq  4}$. Thus, by {\em Theorem \ref{t4}}, we get ${f=g}$. However, if ${Q(z)}$ is constant (say ${C_1}$) then from \eqref{jj4}, we get \begin{eqnarray}\label{jj5}
	\dfrac{f-\dfrac{d_2(1+q_1)+d_1(1+q_2)}{2+q_1+q_2}}{f-d_3}&=&\dfrac{g-\dfrac{d_2(1+q_1)+d_1(1+q_2)}{2+q_1+q_2}}{g-d_3}.
\end{eqnarray}
\par Now, we claim that ${d_3\neq\dfrac{d_2(1+q_1)+d_1(1+q_2)}{2+q_1+q_2}}$. If possible, let ${d_3=\dfrac{d_2(1+q_1)+d_1(1+q_2)}{2+q_1+q_2}}$.
Since ${Q(z)=C_1}$ and ${q_3=1}$, we have \begin{eqnarray}\label{jj7}
	P'(z)&=&C_1(z-d_1)^{q_1}(z-d_2)^{q_2}\left({z-\dfrac{d_2(1+q_1)+d_1(1+q_2)}{2+q_1+q_2}}\right).
\end{eqnarray}  
Integrating \eqref{jj7}, we have \begin{eqnarray}\label{jj8}
	P(z)&=&\dfrac{C_1}{2+q_1+q_2}(z-d_1)^{q_1+1}(z-d_2)^{q_2+1}+C_2,
\end{eqnarray}
where ${C_2}$ is integrating constant. Therefore, ${P(d_1)=P(d_2)=C_2}$, which is a contradiction.\par Thus \eqref{z4}, \eqref{jj5} and the fact that ${P(f)=P(f)}$ together implies ${f,g}$ share five distinct values ${d_1,d_2,d_3,\dfrac{d_2(1+q_1)+d_1(1+q_2)}{2+q_1+q_2},  \infty }$ CM. Therefore, by the Nevanlinna's five value theorem \cite{nevanlinna},  we get ${f=g}$.
\par \textbf{\underline{Case-2.2:}} Let ${P(z)}$ be NCIP. As per our assumption \begin{eqnarray}
	\nonumber \dfrac{Q'(d_3)}{Q(d_3)}&\neq&\dfrac{d_2(1+q_1)+d_1(1+q_2)-(2+q_1+q_2)d_3}{(d_3-d_1)(d_3-d_2)},
\end{eqnarray}
numerator of \eqref{jj4} (in the left hand side) can not vanish at zeros of ${(f-d_3)}$. Again, if ${f(z_0)=d_3}$, then ${Q(g(z_0))\neq 0}$, because otherwise, if we assume $Q(g(z_0))=0$, then ${P(f)=P(g)}$ implies ${P(d_3)=P(g(z_0))}$ and it again implies ${d_3}$ and ${g(z_0)}$ lies in the same column (in {\em Table-1}). Again, as ${q_3=1}$, the column of ${d_3}$ (in {\em Table-1}) will contain only ${d_3}$ and no other element. Thus the only possibility is ${d_3=g(z_0)}$, which is again a contradiction as ${Q(d_3)\neq 0}$. Therefore, from \eqref{jj4}, we see that ${f,g}$ share ${d_3}$ CM. Thus, the proof of the proposition is completed.
\par \vspace{0.1in} So, we find that ${f,g}$ share ${d_1,d_2, \infty }$ CM  and ${d_3}$ \textbf{``CM"}. By the four value theorem \cite{nevanlinna}, we get ${g=T(f)}$, where ${T}$ is a Mobius transformation with the following properties.
\begin{itemize}
	\item [(I)] Two among ${d_1,d_2,d_3, \infty }$ are e.v.p. of both ${f,g}$.\item[(II)] Two among ${d_1,d_2,d_3, \infty }$,(which are e.v.p. of ${f,g}$) interchange each other under ${T}$. 
\end{itemize} Let \begin{eqnarray}\label{z28}
	T(z)&=&\dfrac{az+b}{cz+d},
\end{eqnarray}
where ${ad-bc\neq 0}$. Now ${P(f)=P(g)}$ implies \begin{eqnarray}\label{z29}
	P(f)&=&P\left({\dfrac{af+b}{cf+d}}\right)\\\nonumber a_nf^n+a_{n-1}f^{n-1}+ \dots+a_1f&=&a_n\left({\dfrac{af+b}{cf+d}}\right)^n+a_{n-1}\left({\dfrac{af+b}{cf+d}}\right)^{n-1}\\\nonumber&&+ \dots+a_1\left({\dfrac{af+b}{cf+d}}\right)\\\nonumber(cf+d)^n\left({a_nf^n+ \dots+a_1f}\right)&=& a_n(af+b)^n+a_{n-1}(af+b)^{n-1}(cf+d)\\\nonumber&&+ \dots+ a_1(af+b)(cf+d)^{n-1}.
\end{eqnarray}
Therefore, we must have ${(\text{coefficient of }f^{2n})=0}$; i.e., ${c^na_n=0\Longrightarrow c=0}$. Thus we get \begin{eqnarray}\label{z30}
	g&=&\dfrac{a}{d}f+\dfrac{b}{d}.
\end{eqnarray}
\begin{proposition}\label{propo3}
	${ \infty }$ is an e.v.p. of both ${f}$ and ${g}$.
\end{proposition}
\par\noindent\textit{Proof.} If possible, let ${ \infty}$ is not an e.v.p. of both ${f,g}$. Then, exactly two of ${d_1,d_2,d_3}$ must be e.v.p. of both ${f,g}$. Without loss of generality, let us assume that ${d_1,d_2}$ both are e.v.p. of ${f}$ and $g$. Now, by the second fundamental theorem, we get\begin{eqnarray}\label{mm49}
	2T(r,f)&\leq &  \sum\limits_{j=1}^{3} \overline{N}(r,d_j;f)+ \overline{N}(r, \infty;f)+S(r,f)\\\nonumber&=& \overline{N}(r,d_3;f)+ \overline{N}(r, \infty;f)+S(r,f).
\end{eqnarray}
So, we must have \begin{eqnarray}\label{mm50}
	 \overline{N}(r,d_3;f)&=&T(r,f)+S(r,f)
\end{eqnarray}
and\begin{eqnarray}\label{mm51}
	 \overline{N}(r, \infty;f)&=&T(r,f)+S(r,f).
\end{eqnarray}
Now, define \begin{eqnarray}\label{mm52}
	\eta&=&\dfrac{f-d_3}{g-d_3}.
\end{eqnarray}
Since ${f}$ is non-constant, we must have ${\eta\neq 0}$. Again, as ${f,g}$ share ${d_3, \infty}$ CM, we see that ${\eta}$ has no pole. Therefore, from \eqref{mm50} and \eqref{mm51} we get \begin{eqnarray}\label{mm53}
	m(r,\eta)&\leq  &m(r,f)+m(r,d_3;g)+S(r,f)\\\nonumber&=&m(r,f)+m(r,d_3;f)+S(r,f)\\\nonumber&=&S(r,f).
\end{eqnarray}
Hence \begin{eqnarray}\label{mm54}
	T(r,\eta)&=&S(r,f).
\end{eqnarray}
Again, as ${f,g}$ share the set ${ \mathbb{S}}$ CM, by using the second fundamental theorem and \eqref{mm54}, we must have \begin{eqnarray}\label{mm59}
	(n-2)T(r,f)&\leq  & \sum\limits_{j=1}^{n} \overline{N}(r, b_j;f)+S(r,f)\\\nonumber&\leq  & \sum\limits_{j=1}^{n} \sum\limits_{i=1}^{n} \overline{N}\left({r,\dfrac{b_j-d_3}{b_i-d_3};\eta}\right)+S(r,f)\\\nonumber&\leq  &n^2T(r,\eta)+S(r,f)\\\nonumber&=&S(r,f),
\end{eqnarray}
which is a contradiction, as ${n\geq  5}$.  This proves the proposition.\par\vspace{0.1in} Now, {\em Proposition \ref{propo3}} and (I) together implies that exactly two of ${d_1,d_2,d_3}$ are not e.v.p..  Without loss of generality, let us assume that ${d_1,d_2}$ are not e.v.p. and ${d_3}$ is e.v.p. of ${f,g}$. Therefore, from \eqref{z30}, we get\begin{eqnarray}\label{mm56}
	(g-d_3)&=&\dfrac{a}{d}(f-d_3).
\end{eqnarray}
Now by putting ${f(z_0)=g(z_0)=d_1}$ in \eqref{mm56}, we get ${a=d}$ and hence ${f=g}$. This completes the proof of the theorem.
	   \section*{\Large {Proof of Theorem \ref{t6}}}
	   Let us define \begin{eqnarray}\label{xv4}
	   	\Psi&=&\dfrac{f'g'(f-g)^2}{(f-d_1)(f-d_2)(f-d_3)(g-d_1)(g-d_2)(g-d_3)}
	   \end{eqnarray}
	   and \begin{eqnarray}\label{xv5}
	   	\Gamma&=&\dfrac{(Q(f))'}{Q(f)}-\dfrac{(Q{(g)})'}{Q(g)}\\\nonumber&=&\dfrac{f'Q'(f)}{Q(f)}-\dfrac{g'Q'(g)}{Q(g)}.
	   \end{eqnarray}
   Obviously, \begin{eqnarray}\label{xv20}
   	T(r,\Psi)&=&S(r,f).
   \end{eqnarray}
	   Again, as the zeros of ${Q(z)}$ and ${d_1,d_2,d_3}$ lies in different columns of {\em Table -1}, we must have \begin{eqnarray}\label{xv6}
	   	 \overline{N}(r, 0;Q(f)|f'= 0)&\leq &  \overline{N}(r,0;Q(f)|f'=0,f\neq d_1,d_2,d_3,g\neq d_1,d_2,d_3)\\\nonumber&\leq  & \overline{N}(r,0;\Psi)\\\nonumber&\leq  &T(r,\Psi)+O(1)\\\nonumber&=&S(r,f).
	   \end{eqnarray}
	   Similarly, \begin{eqnarray}\label{xv7}
	   	 \overline{N}(r,0;Q(g)|g'=0)&=&S(r,g).
	   \end{eqnarray}
	   Again, ${f'P'(f)=g'P'(g)}$ implies \begin{eqnarray}\label{xv8}
	   	f'(f-d_1)(f-d_2)(f-d_3)Q(f)&=&g'(g-d_1)(g-d_2)(g-d_3)Q(g).
	   \end{eqnarray}
   Again, using \eqref{xv6} and \eqref{xv7} \begin{eqnarray}\label{xv11}
   	 \overline{N}_*(r,0;Q(f),Q(g))&\leq  & \overline{N}_*(r,0;Q(f),Q(g)|f'=0)+\overline{N}_*(r,0;Q(f),Q(g)|f'\neq 0)\\\nonumber&\leq  & \overline{N}(r,0;Q(f)|f'=0)+ \overline{N}_*(r,0;Q(f),Q(g)|f'\neq 0)\\\nonumber&\leq  &S(r,f)+ \overline{N}_*(r,0;Q(f),Q(g)|f'\neq 0,g'=0)\\\nonumber&&+ \overline{N}_*(r,0;Q(f),Q(g)|f'\neq 0,g'\neq 0)\\\nonumber&\leq  &S(r,f)+ \overline{N}(r,0;g'|f\neq d_1,d_2,d_3,g\neq d_1,d_2,d_3)+0\\\nonumber&\leq  &S(r,f)+ \overline{N}(r,0;\Psi)+0\\\nonumber&\leq  &S(r,f)+T(r,\Psi)\\\nonumber&=&S(r,f),
   \end{eqnarray}
where ${ \overline{N}_*(r,0;Q(f),Q(g))}$ denotes the reduced counting function of those zeros of ${Q(f)}$ which are zeros of ${Q(g)}$ but of different multiplicities. 
   Furthermore, from \eqref{xv8} we have \begin{eqnarray}\label{xv15}
   	 \overline{N}(r,0;Q(f)|Q(g)\neq 0)&\leq  & \overline{N}(r,0;g'|f\neq d_1,d_2,d_3,g\neq d_1,d_2,d_3)\\\nonumber&\leq  & \overline{N}(r,0;\Psi)\\\nonumber&\leq  &S(r,f).
   \end{eqnarray}
   Similarly, \begin{eqnarray}\label{xv16}
   	 \overline{N}(r,0;Q(g)|Q(f)\neq 0)&=&S(r,f).
   \end{eqnarray}
   Now, as ${f,g}$ share ${ \infty }$ CM, poles of ${f,g}$ are not the poles of ${\Gamma}$. Again, as all poles of ${\Gamma}$ are simple, we have by using \eqref{xv11}, \eqref{xv15} and \eqref{xv16} \begin{eqnarray}\label{xv10}
	   	N(r,\Gamma)&\leq  & \overline{N}_*(r,0;Q(f),Q(g))+ \overline{N}(r,0;Q(f)|Q(g)\neq 0)\\\nonumber&&+ \overline{N}(r,0;Q(g)|Q(f)\neq 0)\\\nonumber&= & S(r,f).
	   \end{eqnarray}
	   Thus, \begin{eqnarray}\label{xv19}
	   	T(r,\Gamma)&=&N(r,\Gamma)+m(r,\Gamma)\\\nonumber&=&S(r,f).
	   \end{eqnarray}
	   Now by using the second fundamental theorem, we have\begin{eqnarray}
	   	\nonumber T(r,f)&\leq  & \sum\limits_{j=1}^{3} \overline{N}(r,d_j;f)-N_0(r,0;f')+S(r,f)\\\label{xv21}&\leq  & \sum\limits_{j=1}^{3} \overline{N}(r,d_j;f|g'\neq 0)+ \sum\limits_{j=1}^{3} \overline{N}(r,d_j;f|g'=0)-N_0(r,0;f')+S(r,f).
	   \end{eqnarray}
	   Now, using {\em Lemma \ref{lem4}} and {\em Lemma \ref{lem1}}, we get \begin{eqnarray}\label{xv1}
	   	 \sum\limits_{j=1}^{3} \overline{N}(r,d_j;f|g'\neq 0)&= &  \sum\limits_{j=1}^{3} \overline{N}(r,d_j;f|g'\neq 0,g=d_j)\\\nonumber&\leq  &\sum\limits_{j=1}^{3} \overline{N}(r,d_j;f|g=d_j,g'\neq 0,f'\neq 0).
	   \end{eqnarray}
	   Again, \begin{eqnarray}\label{xv2}
	   	\nonumber \sum\limits_{j=1}^{3} \overline{N}(r,d_j;f|g'=0)&\leq  & \sum\limits_{j=1}^{3} \overline{N}(r,d_j;f|g'=0,f'=0)+ \sum\limits_{j=1}^{3} \overline{N}(r,d_j;f|g'=0,f'\neq 0)\\\nonumber&\leq  & \sum\limits_{j=1}^{3} \overline{N}(r,d_j;f|g'=0,f'=0)+ \sum\limits_{j=1}^{3} \overline{N}(r,d_j;f|g'=0,f'\neq 0,g=d_j)\\\nonumber&&+  \sum\limits_{j=1}^{3} \overline{N}(r,d_j;f|g'=0,f'\neq 0,g\neq d_j)\\\label{xv22}&\leq  &\sum\limits_{j=1}^{3} \overline{N}(r,d_j;f|g'=0,f'=0)+ \sum\limits_{j=1}^{3} \overline{N}(r,d_j;f|g'=0,f'\neq 0,g=d_j)\\\nonumber&&+   \overline{N}_0(r,0;g')\\\nonumber&\leq  & \overline{N}(r,0;\Gamma)+0+ \overline{N}_0(r,0;g')\qquad [\text{Using }Lemma\,\,\, \ref{lem1} ]\\\nonumber&= & \overline{N}_0(r,0;g')+S(r,f)\qquad [ \text{ Using }\eqref{xv19}].
	   \end{eqnarray}
	   Thus, by using \eqref{xv1}, \eqref{xv22} in \eqref{xv21}, we get\begin{eqnarray}\label{xv30}
	   	T(r,f)&\leq  &\sum\limits_{j=1}^{3} \overline{N}(r,d_j;f|g=d_j,g'\neq 0,f'\neq 0)\\\nonumber&&+ \overline{N}_0(r,0;g')-N_0(r,0;f')+S(r,f).
	   \end{eqnarray}
	   Similarly, \begin{eqnarray}\label{xv31}
	   	  	T(r,g)&\leq  &\sum\limits_{j=1}^{3} \overline{N}(r,d_j;g|f=d_j,f'\neq 0,g'\neq 0)\\\nonumber&&+ \overline{N}_0(r,0;f')-N_0(r,0;g')+S(r,g).
	   \end{eqnarray}
	   Combining \eqref{xv30} and \eqref{xv31}, we get \begin{eqnarray}\label{xv33}
	   	\left[{T(r,f)+T(r,g)}\right]&\leq  &2 \sum\limits_{J=1}^{3} \overline{N}(r,d_j;f|g=d_j,g'\neq 0,f'\neq 0)+S(r,f).
	   \end{eqnarray} 
   Now let \begin{eqnarray}\label{xv35}
   	f(z)&=&d_1+(z-z_0)A_1+(z-z_0)^2A_2+ \dots \\\nonumber and \quad g(z)&=&d_1+(z-z_0)B_1+(z-z_0)^2B_2+ \dots ,
   \end{eqnarray}
   where ${A_i,B_i}$'s are constants with ${A_1,B_1(\neq 0)}$. Then\begin{eqnarray}\label{xv36}
   	\Psi(z_0)&=&\dfrac{(A_1-B_1)^2}{(d_1-d_2)^2(d_1-d_3)^2} \\\nonumber  \text{ and} \quad \Gamma(z_0)&=&\dfrac{Q'(d_1)}{Q(d_1)}\left({A_1-B_1}\right).
   \end{eqnarray}
   So, \begin{eqnarray}\label{xv37}
   	\Gamma(z_0)^2&=&\left({\dfrac{Q'(d_1)(d_1-d_2)(d_1-d_3)}{Q(d_1)}}\right)^2\Psi(z_0)\\\nonumber&=&c_1\Psi(z_0),
   \end{eqnarray}
   where ${c_1=\left({\dfrac{Q'(d_1)(d_1-d_2)(d_1-d_3)}{Q(d_1)}}\right)^2}$. Notice that instead of taking simple ${d_1}$ point of ${f,g}$ at ${z_0}$ in \eqref{xv35}, if we consider simple ${d_2( \text{or } d_3)}$ point of both ${f,g}$, we would get \begin{eqnarray}\label{xv40}
   	\Gamma(z_0)^2&=&c_2 \Psi(z_0)\\\label{xv41}  \text{ and }\quad \Gamma(z_0)^2&=&c_3\Psi(z_0),
   \end{eqnarray}
   where \begin{eqnarray}\label{xv42}
   	c_i&=&\left({\dfrac{Q'(d_i)(d_i-d_j)(d_i-d_k)}{Q(d_i)}}\right)^2, 
   \end{eqnarray}
   for distinct ${i,j,k\in \{1,2,3\}}$. We ensure that at least one of ${\Gamma^2=c_i\Psi}$ for ${i=1,2,3}$ holds. Because, if ${\Gamma^2\neq c_i\Psi}$ for each ${i=1,2,3}$, then \begin{eqnarray}\label{xv43}
  \sum\limits_{J=1}^{3} \overline{N}(r,d_j;f|g=d_j,g'\neq 0,f'\neq 0)&\leq  &  \sum\limits_{j=1}^{3}\overline{N}(r,0;\Gamma^2-c_j\Psi)\\\nonumber&\leq  &6T(r,\Gamma)+3T(r,\Psi)+O(1)\\\nonumber&=&S(r,f).
   \end{eqnarray}
Hence from \eqref{xv33}, we get \begin{eqnarray}
	\nonumber T(r,f)+T(r,g)=S(r,f),
\end{eqnarray}
which is a contradiction. Therefore, without loss of generality, we may assume \begin{eqnarray}\label{xv50}
	\Gamma^2&=&c_1\Psi.
\end{eqnarray}
Now 
   \begin{eqnarray}\label{xv51}
   	&&\left(Q(f)^2Q(g)^2\prod_{j=1}^{3}(f-d_j)(g-d_j)\right)\left(\Gamma^2-c_1\Psi\right)\\\nonumber=&&  f'^2 Q'(f)^2 Q(g)^2  \prod\limits_{j=1}^{3}(f-d_j) (g-d_j)\\\nonumber&&-Q(f)Q(g)f'g'\left[{2Q'(f)Q'(g) \prod\limits_{i=1}^{3}(f-d_i)(g-d_i)+c_1(f-g)^2Q(f)Q(g)}\right]\\\nonumber&&+g'^2Q'(g)^2Q(f)^2 \prod\limits_{i=1}^{3}(f-d_i)(g-d_i).
   \end{eqnarray}
Now using ${\dfrac{f'}{g'}=\dfrac{P'(g)}{P'(f)}}$, we have \begin{eqnarray}\label{xv60}
	&&\dfrac{Q(f)^2Q(g)^2\prod_{j=1}^{3}(f-d_j)(g-d_j)}{g'^2}(\Gamma^2-c_1\Psi)\\\nonumber=&&\left({\dfrac{P'(g)}{P'(f)}}\right)^2Q'(f)^2Q(g)^2 \prod\limits_{j=1}^{3}(f-d_i)(g-d_i)\\\nonumber&&-Q(f)Q(g)\left({\dfrac{P'(g)}{P'(f)}}\right)\left[{2Q'(f)Q'(g) \prod\limits_{i=1}^{3}(f-d_i)(g-d_i)+c_1(f-g)^2Q(f)Q(g)}\right]\\\nonumber&&+Q'(g)^2Q(f)^2 \prod\limits_{i=1}^{3}(f-d_i)(g-d_i).
\end{eqnarray}
Now, replacing ${\dfrac{P'(g)}{P'(f)}}$ by ${\dfrac{Q(g) \prod\limits_{j=1}^{3}(g-d_j)}{Q(f) \prod\limits_{j=1}^{3}(f-d_j)}}$ and then simplifying, we get \begin{eqnarray}\label{xv70}
	&&\dfrac{Q(f)^2Q(g)^2\prod_{j=1}^{3}(f-d_j)(g-d_j)}{g'^2}\left({\Gamma^2-c_1\Psi}\right)\\\nonumber=&&\dfrac{\prod\limits_{j=1}^{3} (g-d_j)}{ \prod\limits_{j=1}^{3}(f-d_j)}\left[\dfrac{Q(g)^4Q'(f)^2}{Q(f)^2} \prod\limits_{j=1}^{3}(g-d_i)^2+Q(f)^2Q'(g)^2 \prod\limits_{j=1}^{3}(f-d_j)^2\right.\\\nonumber&&-Q(g)^2\left.\left({2Q'(f)Q'(g) \prod\limits_{j=1}^{3}(f-d_j)(g-d_j)+c_1Q(f)Q(g)(f-g)^2}\right)\right]
\end{eqnarray}
Again, as  ${\dfrac{\prod\limits_{j=1}^{3} (g-d_j)}{ \prod\limits_{j=1}^{3}(f-d_j)}\neq 0}$ and $\Gamma ^2-c_1\Psi=0$, so ${\dfrac{Q(f)^2Q(g)^2\prod_{j=1}^{3}(f-d_j)(g-d_j)}{g'^2}\left({\Gamma^2-c_1\Psi}\right)=0}$. Thus from \eqref{xv70}, we have \begin{eqnarray}\label{xv71}
	&&\dfrac{Q(g)^4Q'(f)^2}{Q(f)^2} \prod\limits_{j=1}^{3}(g-d_i)^2+Q(f)^2Q'(g)^2 \prod\limits_{j=1}^{3}(f-d_j)^2\\\nonumber=&&Q(g)^2\left({2Q'(f)Q'(g) \prod\limits_{j=1}^{3}(f-d_j)(g-d_j)+c_1Q(f)Q(g)(f-g)^2}\right).
\end{eqnarray}
Let ${f(z_0)=d_1}$ and ${g(z_0)=d}$. So, ${Q(d)Q(d_1)\neq 0}$. Now putting ${z=z_0}$ in \eqref{xv71}, we have  \begin{eqnarray}\label{xv77}
	\dfrac{Q(d)^3}{Q(d_1)^2}(d-d_1)^2\left[{Q(d)Q'(d_1)^2 \prod\limits_{j=2}^{3}(d-d_j)^2-c_1Q(d_1)^3}\right]&=&0.
\end{eqnarray}
Now, since ${f(z_0)=d_1}$, so ${P(f)=P(g)}$ implies  ${d\in \{z:P(z)-P(d_1)=0\}}$. Now using the value of ${c_1}$ from \eqref{xv42}, we have
\begin{eqnarray}\label{xv79}
	&&{Q(d)Q'(d_1)^2 \prod\limits_{j=2}^{3}(d-d_j)^2-c_1Q(d_1)^3}\\\nonumber=&&Q'(d_1)^2\left[{Q(d)(d-d_2)^2(d-d_3)^2-Q(d_1)(d_1-d_2)^2(d_1-d_3)^2}\right]\\\nonumber\neq &&0 \text{ (as per our assumption) }.
\end{eqnarray}
Therefore, from \eqref{xv77}, we must have \begin{eqnarray}\label{xv80}
	d=d_1.
\end{eqnarray}
Again, if we consider ${g(z_0)=d_1}$ and ${f(z_0)=d}$, then putting ${z=z_0}$ in \eqref{xv71} and proceeding similarly, we get ${d=d_1}$. Thus ${f,g}$ share ${d_1}$. Again, by {\em Lemma \ref{lem1}}, we get ${f,g}$ share ${d_1}$ CM. \par\vspace{0.1in} Notice that while proving {\em Theorem \ref{t7}}, ${q_l= \max  (q_1,q_2,q_3)\geq  2}$ was used just to show (with the help of {\em Lemma \ref{lem6}}) that ${f,g}$ share ${d_l}$ \textbf{``CM"} (see \eqref{x9} and \eqref{x11}).
 \par \vspace{0.1in} Here, at this point of our proof of {\em Theorem \ref{t6}}, we have ${f,g}$ share ${d_1}$ CM. Therefore, the remaining proof of the {\em Theorem \ref{t6}} is similar to the proof of {\em Theorem \ref{t7}}. Just notice that in our case ${P(z)}$ is NCIP, so {\em Case-2.1} in {\em Theorem \ref{t7}} will be redundant. 
\section*{\Large Proof of Theorem \ref{themm9}}
Let \begin{eqnarray}\label{xx1}
	P'(z)&=&(z-d_1)^{q_1}(z-d_2)^{2}(z-d_3)^{q_3},
\end{eqnarray}
such that ${t=2,t'=3}$, and ${B_1(H_2)=\{d_1,d_3\}, B_2(H_2)=\{d_2\}}$ with ${q_3<q_1}$. Therefore, ${P(d_1)=P(d_3)}$ and hence \begin{eqnarray}\label{kkl20}
	P(z)-P(d_1)&=&c(z-d_1)^{q_1+1}(z-d_3)^{q_3+1}(z- \alpha ),
\end{eqnarray}
where ${c\neq 0}$ and ${ \alpha \neq d_1,d_2,d_3}$.
Since ${P(f)=P(g)}$, we have \begin{eqnarray}\label{kkl21}
	P(f)-P(d_1)&=&P(g)-P(d_1)\\\nonumber \Longrightarrow  \quad (f-d_1)^{q_1+1}(f-d_3)^{q_3+1}(f- \alpha)&=&(g-d_1)^{q_1+1}(g-d_3)^{q_3+1}(g- \alpha);
\end{eqnarray}
and differentiating ${P(f)=P(g)}$, we have \begin{eqnarray}\label{kkl22}
	f'P'(f)&=&g'P'(g)\\\nonumber \Longrightarrow  \quad f'(f-d_1)^{q_1}(f-d_2)^{2}(f-d_3)^{q_3}&=&g'(g-d_1)^{q_1}(g-d_2)^{2}(g-d_3)^{q_3}.
\end{eqnarray}
From \eqref{kkl21} \begin{eqnarray}\label{kkl23}
	\overline{N}(r, \alpha ;f)&=& \overline{N}(r, \alpha ;f|g= \alpha)+ \overline{N}(r, \alpha;f|g\neq  \alpha )\\\nonumber&\leq  & \overline{N}(r, \alpha;f|g= \alpha)+\dfrac{1}{q_3+1}N(r, \alpha;f).
\end{eqnarray}
Similarly,\begin{eqnarray}\label{kkl24}
	\overline{N}(r, \alpha ;g)&\leq  & \overline{N}(r, \alpha ;g|f= \alpha )+\dfrac{1}{q_3+1}N(r, \alpha;g).
\end{eqnarray}
\par Again, using \eqref{kkl21} and {\em Lemma \ref{lem4}}, we have
\begin{eqnarray}\label{kkl26}
	\overline{N}(r,d_1;f|g\neq d_1)&\leq  & \overline{N}(r,d_1;f|g'=0,g=d_3)+ \overline{N}(r,d_1;f|g'=0,g= \alpha ).
\end{eqnarray}
\par\vspace{0.1in}\noindent\textbf{Claim-1:} ${ \overline{N}(r,d_1;f|g'=0,g=d_3)\leq  \dfrac{q_3+1}{2(q_1+1)}N(r,d_3;g)}$.
\par\noindent\textit{Proof.} Let ${f(z_0)=d_1}$ with multiplicity ${p}$ and ${g(z_0)=d_3}$ with multiplicity ${q(\geq  2)}$. Since ${q_1<2q_3+1}$, from \eqref{kkl22}, we see that ${p\neq 1}$. Hence ${p\geq  2}$. So, \begin{eqnarray}
	\nonumber p(q_1+1)&=&q(q_3+1)\\\nonumber i.e.,\quad q&=&\dfrac{p(q_1+1)}{q_3+1}\\\nonumber&\geq  &\dfrac{2(q_1+1)}{q_3+1}.
\end{eqnarray}
This proves {\em Claim-1}. 
\par\vspace{0.1in}\noindent\textbf{Claim-2:} ${ \overline{N}(r,d_1;f|g'=0,g= \alpha)\leq  \dfrac{1}{q_1+1}N(r, \alpha;g)}$. 
\par\noindent\textit{Proof.} It follows directly from \eqref{kkl21}.\par\vspace{0.2in} Therefore, using {\em Claim-1,Claim-2}, in \eqref{kkl26}, we get \begin{eqnarray}\label{kkl261}
	\overline{N}(r,d_1;f|g\neq d_1)&\leq  &\dfrac{q_3+1}{2(q_1+1)}N(r,d_3;g)+\dfrac{1}{q_1+1}N(r, \alpha;g).
\end{eqnarray}
Again, \begin{eqnarray}\label{kkl27}
	\overline{N}(r,d_3;f|g\neq d_3)&\leq & \overline{N}(r, d_3;f|g'=0,g=d_1)+ \overline{N}(r,d_3;f|g'=0,g= \alpha).
\end{eqnarray}
\par\noindent\textbf{Claim-3:} ${ \overline{N}(r,d_3;f|g'=0,g=d_1)\leq  \dfrac{q_3+1}{2(q_1+1)}N(r,d_3;f)}$.\par\noindent \textit{Proof.} Let ${f(z_0)=d_3}$ of multiplicity ${p}$ and ${g(z_0)=d_1}$ of multiplicity ${q(\geq  2)}$. Therefore, from \eqref{kkl21} \begin{eqnarray}
	\nonumber p(q_3+1)&=&q(q_1+1)\\\nonumber p&=&\dfrac{q(q_1+1)}{q_3+1}\\\nonumber&\geq  &\dfrac{2(q_1+1)}{q_3+1}.
\end{eqnarray}
Hence {\em Claim-3} follows.
\par\noindent\textbf{Claim-4:} ${ \overline{N}(r,d_3;f|g'=0,g= \alpha)\leq  \dfrac{1}{q_3+1}N(r, \alpha;g)}$. \par\noindent\textit{Proof.} Directly follows from \eqref{kkl21}.\par\vspace{0.2in} Therefore, using {\em Claim-3, Claim-4}, in \eqref{kkl27}, we get 
\begin{eqnarray}\label{kkl28}
	\overline{N}(r,d_3;f|g\neq d_3)&\leq  &\dfrac{q_3+1}{2(q_1+1)}N(r,d_3;f)+\dfrac{1}{q_3+1}N(r, \alpha;g).
\end{eqnarray}
\par\noindent\textbf{Claim-5:} ${ \overline{N}(r,d_2;f|g\neq d_2)\leq   \overline{N}_0(r,0;g'|f=d_2)}$.\par\vspace{0.1in}\noindent\textit{Proof.} Let ${f(z_0)=d_2,g(z_0)\neq d_2}$. Then, from {\em Lemma \ref{lem4}}, we must have ${g'(z_0)=0}$. Again, as ${P(d_1)=P(d_3)=P( \alpha)}$ and ${P(d_2)\neq P(d_1)}$, from ${P(f)=P(g)}$, we get ${g(z_0)\neq d_1,d_2,d_3, \alpha}$. Hence {\em Claim-5} holds.
\par\vspace{0.2in} Let us define \begin{eqnarray}\label{kkl29}
	\xi&=&\dfrac{1}{f-C}-\dfrac{1}{g-C},
\end{eqnarray}
where ${C(\neq d_1,d_2,d_3, \alpha)\in  \mathbb{C}}$.
\par\vspace{0.1in}\textbf{Case-1:} Let ${\xi=0}$. Then ${f=g}$.
\par\vspace{0.1in}\textbf{Case-2:} Let ${\xi\neq 0}$.
Applying the second fundamental theorem and then using \eqref{kkl26}, \eqref{kkl28}, \eqref{kkl23},  {\em Claim-5} and ${T(r,g)=T(r,f)+S(r,f)}$, we have\begin{eqnarray}\label{kkl30}
	3T(r,f)&\leq  & \sum\limits_{j=1}^{3} \overline{N}(r,d_j;f)+ \overline{N}(r, \alpha;f)+ \overline{N}(r, \infty;f)-N_0(r,0;f')+S(r,f)\\\nonumber&\leq  & \left[{\sum\limits_{j=1}^{3} \overline{N}(r,d_j;f|g=d_j)+ \overline{N}(r, \alpha;f|g= \alpha)}\right]\\\nonumber&&+\left[{\sum\limits_{j=1}^{3} \overline{N}(r,d_j;f|g\neq d_j)}+\dfrac{1}{q_3+1}N(r, \alpha;f) \right]\\\nonumber&&+ \overline{N}(r, \infty;f)-N_0(r,0;f')+S(r,f)\\\nonumber&\leq  &N(r,0;\xi)+\dfrac{q_3+1}{2(q_1+1)}N(r,d_3;g)+\dfrac{q_3+1}{2(q_1+1)}N(r,d_3;f)\\\nonumber&&+\left({\dfrac{1}{q_1+1}+\dfrac{1}{q_3+1}}\right)N(r, \alpha;g)+ \overline{N}_0(r,0;g')\\\nonumber&&-N_0(r,0;f')+S(r,f)\\\nonumber&\leq  &\left({2+\dfrac{q_3+1}{q_1+1}+\dfrac{1}{q_1+1}+\dfrac{1}{q_3+1}}\right)T(r,f)\\\nonumber&&+ \overline{N}_0(r,0;g')-N_0(r,0;f')+S(r,f).
\end{eqnarray}
Similarly, \begin{eqnarray}\label{kkl31}
	3T(r,g)&\leq  &\left({2+\dfrac{q_3+1}{q_1+1}+\dfrac{1}{q_1+1}+\dfrac{1}{q_3+1}}\right)T(r,g)\\\nonumber&&+ \overline{N}_0(r,0;f')-N_0(r,0;g')+S(r,g).
\end{eqnarray}
Adding \eqref{kkl30} and \eqref{kkl31}, 
we have \begin{eqnarray}
	\nonumber 3[T(r,f)+T(r,g)]&\leq  &\left({2+\dfrac{q_3+1}{q_1+1}+\dfrac{1}{q_1+1}+\dfrac{1}{q_3+1}}\right)[T(r,f)+T(r,g)]+S(r,f)+S(r,g),
\end{eqnarray}
which is a contradiction if \begin{eqnarray}
	\label{kkl32} 3&>&\left({2+\dfrac{q_3+1}{q_1+1}+\dfrac{1}{q_1+1}+\dfrac{1}{q_3+1}}\right).
\end{eqnarray}
Again, we have \begin{eqnarray}\label{kkl33}
	q_3<q_1<2q_3+1.
\end{eqnarray}
Solving \eqref{kkl32} and \eqref{kkl33}, we get \begin{eqnarray}
	\nonumber q_1>5 \quad\text{ and }\quad  \dfrac{q_1-1}{2}<q_3<\dfrac{q_1-2}{2}+\dfrac{\sqrt{q_1^2-4q_1-4}}{2}.
\end{eqnarray}
This completes the proof of the theorem. 
 \section*{\Large{Proof of Theorem \ref{tt7}}}
 Let ${P(z)=z^4+a_3z^3+a_2z^2+a_1z+a_0}$ be a four degree UPE. Then from {\em Theorem C}, we get\begin{eqnarray}\label{vg1}
 	\dfrac{a_3^3}{8}-\dfrac{a_2 a_3}{2}+a_1 &\neq& 0.
 \end{eqnarray}
 We shall show that ${P(z)}$ is CIP. On the contrary, let ${P(z)}$ is NCIP. Then, there exist ${x}$ and ${y}$, such that \begin{eqnarray}\label{vg2}
 	P'(x)&=&0,\\\label{vg3} P'(y)&=&0,\\\label{vg4}P(x)&=&P(y),
 \end{eqnarray}
 with ${x\neq y}$. Now eliminating ${x}$ and ${y}$ from \eqref{vg2}-\eqref{vg4}, we get ${\dfrac{a_3^3}{8}-\dfrac{a_2 a_3}{2}+a_1=0}$, which contradicts \eqref{vg1}. Therefore, ${P(z)}$ is CIP.\par Conversely, let ${P(z)}$ be a  CIP with derivative index greater equal to two. We shall show that ${P(z)}$ is UPE. That is, we have to show \eqref{vg1} holds. On the contrary,  let
 \begin{eqnarray}\label{vg5}
 	\dfrac{a_3^3}{8}-\dfrac{a_2 a_3}{2}+a_1&=&0.
 \end{eqnarray}
 Then \begin{eqnarray}\label{vg6}
 	P(z)&=&z^4+a_3z^3+a_2z^2+\left({\dfrac{a_2a_3}{2}-\dfrac{a_3^3}{8}}\right)z+a_0.
 \end{eqnarray}
 Therefore, solving ${P'(z)=0}$, we get ${z_1=-\dfrac{a_3}{4}}$, ${z_2=\dfrac{1}{4} \left(-\sqrt{3 a_3^2-8 a_2}-a_3\right)}$, and \\ ${z_3=\dfrac{1}{4} \left(\sqrt{3 a_3^2-8 a_2}-a_3\right)}$. Since derivative index is greater equal to two, we must have ${z_2\neq z_3}$. Now, one can verify that \begin{eqnarray}
 	\nonumber P(z_2)=P(z_3)&=&a_0-\frac{1}{64} \left(a_3^2-4 a_2\right)^2,
 \end{eqnarray}
 which is a contradiction as ${P(z)}$ is CIP.\par This completes the proof of the theorem.
 \section*{\Large{Proof of Theorem \ref{tt8}}}
 Here \begin{eqnarray}\label{vb*}
 	P(z)&=&z^5+az^4+bz^3+c.
 \end{eqnarray}
 Since ${P(z)}$ is NCIP, there exist ${x}$ and ${y}$ such that \begin{eqnarray}\label{vb1}
 		P'(x)&=&0,\\\label{vb2} P'(y)&=&0,\\\label{vb3}P(x)&=&P(y).
 \end{eqnarray}
 Eliminating ${x}$ and ${y}$ from \eqref{vb1}-\eqref{vb3}, we get ${a\neq 0}$ and \begin{eqnarray}\label{vb4}
 \left(a^2-4 b\right) \left(8 a^2-5 b\right)&=&0.
 \end{eqnarray}
\par Since ${8a^2\neq 5b}$, from \eqref{vb4}, we must have ${a^2=4b}$, and hence ${P(z)=z^5+az^4+\dfrac{a^2}{4}z^3+c}$. Let \begin{eqnarray}\label{vb6}
	f_1&=&\frac{-ae^{3 z}-a \left(e^{2 z}+1\right) \left(e^{4 z}+1\right)}{2 \left(e^{2 z}+e^{4 z}+e^{6 z}+e^{8 z}+1\right)}\quad  \text{ and }\quad g_1=e^{2z} f_1.
\end{eqnarray}
Now, one can verify that ${P(f_1)=P(g_1)}$ but ${f_1\neq g_1}$.\par This completes the proof of the theorem.
      	 \section{Acknowledgement} 
      	     Sanjay Mallick is thankful to``Science and Engineering Research Board, Department of Science and Technology, Government of India"  for financial support to pursue this research work under the Project File No. EEQ/2021/000316. 	
      	     \section{Competing interests}
      	     The authors have no relevant financial or non-financial interests to disclose.  	  		
 

\begin{thebibliography}{99}
    \bibitem{An-Complex Variable} T. T. H. An, J.T.-Y. Wang and P.-M. Wong, Strong uniqueness polynomials: the complex case, Complex Variables, Theory and Application: An International Journal 49.1 (2004), 25-54.
	\bibitem{An-11}T. T. H. An, Unique range sets for meromorphic functions constructed without an injectivity hypothesis, Taiwanese J. Math., 15(2) (2011), 697-709.
        \bibitem{ban12}A. Banerjee and I. Lahiri, A uniqueness polynomial generating a unique range set and vise versa, Comput. Methods Funct. Theory, 12(2) (2012), 527-539.
        	\bibitem{Mallick-cmft} A. Banerjee and S. Mallick, On the characterisations of a new class of strong uniqueness polynomials generating unique range sets, Comput. Methods Funct. Theory, 17 (2017), 19-45 (DOI: 10.1007/s40315-016-0174-y).
            \bibitem{Frank Rainders} G. Frank and M. Reinders, A unique range set for meromorphic functions with 11 elements. Complex Var. Theory Appl., 37 (1) (1998), 185-193.
	\bibitem{p-adBou-90}  A. Boutabaa, “Th´eorie de Nevanlinna p-adique,” Manuimoto, On uniqueness for meromorphic functions sharing finite sets, in ``Proceedings ISAAC Conscr. Math. 67, 251–269 (1990).
    \bibitem{IMathBoutabaa} A. Boutabaa, A. Escassut and L. Hadded, On existance of p-adic entire functions, Indag. Mathem., N.S., 8(2), 145-155.

    \bibitem{fuji-2000} H. Fujimoto, On uniqueness for meromorphic functions sharing finite sets, in ``Proceedings ISAAC Congress of Fukuoka, August 1999,'' in press.
	\bibitem{fuji-03}  H. Fujimoto, On uniqueness polynomials for meromorphic functions, Nagoya Math. J., 170 (2003), 33-46.
	\bibitem{gross} F. Gross, Factorization of meromorphic functions and some open problems, Proc. Conf. Univ. Kentucky, Leixngton, Kentucky(1976); Lecture Notes in Math., 599 (1977), 51-69, Springer(Berlin)
	\bibitem{heyman} W. K. Hayman, Meromorphic Functions, The Clarendon Press, Oxford (1964)
	\bibitem{p-adkhai-88} Ha Huy Khoai and My Vinh Quang, “On p-adic Nevanlinna theory,” Lect. Notes Math. 1351, 146-158 (1988).
	\bibitem{ccy-kodai95}   P. Li and C. C. Yang,  Some further results on the unique range sets of meromorphic functions. - Kodai Math. J. 18, 1995, 437–450.

	\bibitem{mallick-filo}  S. Mallick, Unique Range Sets - A Further Study, Filomat 34:5 (2020), 1499–1516 (https://doi.org/10.2298/FIL2005499M) \bibitem{mallick-tbi}  S. Mallick, On the existence of unique range sets generated by non-critically injective polynomials and related issues Sanjay, Tbilisi Mathematical Journal 13(4) (2020),  81–101.
	\bibitem{mues-89}  E. Mues, Meromorphic functions sharing four values, Complex Variables, 12(1989), 169-179
	\bibitem{nevanlinna} R. Nevanlinna, Einige Eindeutigkeitssa¨tze in der Theorie der Meromorphen Funktionen, Acta Math. 48 (1926), no. 3–4, 367–391.
    \bibitem{Ripan} R. Saha and S. Mallick, Unique range sets: a further study II, Rendiconti del Circolo Matematico di Palermo Series 2 (2022),  1-26.

	\bibitem{taipei}   J. Wang, Uniqueness polynomials and bi-unique range sets for rational functions and non Archimedean meromorphic functions,  Acta Arithm. 104 (2), 183-200 (2002).
        \bibitem{c.c.y. book} C. C. Yang and H. X. Yi, Uniqueness theory of meromorphic functions, Kluwer Academic Publishers, Dordrecht, 2003    
    \bibitem{Hxy-Bul} H. X. Yi, Unicity theorems for meromorphic or entire functions III, Bull. Austral. Math. Soc., 53 (1996), 71-82.
\end{thebibliography}
\end{document}